\documentclass[a4paper,11pt]{article} 

\usepackage{amsmath} 
\usepackage{amssymb} %letters for number sets and indicator function
\usepackage{amsthm} %theorems, propositions, proofs
\usepackage{bbm} % for indicator functions
\usepackage{tikz}

\numberwithin{equation}{section}%NUMBERING EQUATIONS
\usepackage{multicol}
\usepackage[backend=bibtex,maxnames=15]{biblatex} %bibliography
\bibliography{motifs_bib}

\AtEveryBibitem{%
    \clearfield{url}%  %these commands erase some fields to display a shorte bibliography
    \clearfield{urldate}%
    \clearfield{doi}%
		\clearfield{eprint}%
		\clearfield{ISSN}
}
\usepackage{graphicx} 
\usepackage{subcaption} %FOR SUBFIGURES
\usepackage{framed}
\usepackage{wrapfig}

\usepackage{enumerate} %use enumerate 
\usepackage{enumitem}
\setlist[enumerate]{itemsep = -0.2em}

\usepackage{authblk} %package for email and affiliations of authors

\usepackage[explicit]{titlesec}  %modifiy the structure of subsection title
\titleformat{\subsection}[runin]
  {\normalfont\normalsize\bfseries}{\thesubsection}{0.3em}{#1.}
	
\titleformat{\section}
  {\normalfont\large}{\thesection}{0.2em.}{\centering \textsc{#1}}

\usepackage{caption}  %package to restyle the captions
\usepackage{hyperref} %for hyperlinks
\usepackage{color} %package for colors
\definecolor{MyDarkBlue}{rgb}{0,0.08,0.50}  %to have colored hyperlinks
\definecolor{BrickRed}{rgb}{0.65,0.08,0}

\hypersetup{  %command that sets the hyperlinks color
colorlinks=true,       % false: boxed links; true: colored links
    linkcolor=MyDarkBlue,          % color of internal links
    citecolor=BrickRed,        % color of links to bibliography
    filecolor=red,      % color of file links
    urlcolor=cyan           % color of external links
}

\usepackage{a4wide}
\allowdisplaybreaks

\usepackage{fancyhdr}
\newtheorem{Lemma}{Lemma}[section]
\newtheorem{lemma}{Lemma}[section]
\newtheorem{Proposition}[Lemma]{Proposition}

\newtheorem{Theorem}[Lemma]{Theorem}

\newtheorem{Remark}[Lemma]{Remark}

\newtheorem{Corollary}[Lemma]{Corollary}
\newtheorem{Definition}[Lemma]{Definition}

\newcommand{\sub}[1]{\boldsymbol{#1}}

\newcommand{\N}{\mathbb{N}}

\newcommand{\I}{\mathbbm{1}}
\newcommand{\pr}{\mathbb{P}}
\newcommand{\E}{\mathbb{E}}

\newcommand{\Ee}{\mathcal{E}}

\newcommand{\eqn}[1]{\begin{equation} #1 \end{equation}}
\newcommand{\eqan}[1]{\begin{align} #1 \end{align}}

\newcommand{\PA}{\mathrm{PA}}

\newcommand{\PU}{\mathrm{PU}}
\newcommand{\din}{d^{\scriptscriptstyle(\mathrm{in})}}
\newcommand{\dout}{d^{\scriptscriptstyle(\mathrm{out})}}
\newcommand{\sss}{\scriptscriptstyle}

\newcommand{\prob}{\mathbb{P}}
\newcommand{\Prob}[1]{\prob\left(#1\right)}
\newcommand{\Probp}[1]{\prob_{\psi_t}\left(#1\right)}

\newcommand{\expec}{\mathbb{E}}
\newcommand{\Exp}[1]{\expec\left[#1\right]}
\newcommand{\Expd}[1]{\expec_{\psi_t}\left[#1\right]}

\newcommand{\Varp}[1]{\textup{Var}_{\psi_t}\left(#1\right)}

\newcommand{\plim}{\ensuremath{\stackrel{\prob}{\longrightarrow}}}
\newcommand{\dlim}{\ensuremath{\stackrel{d}{\longrightarrow}}}

\newcommand{\bigO}[1]{O\left(#1\right)}

\newcommand{\dd}{{\rm d}}

\graphicspath{{Pictures/}{Pictures/motifs}}

\begin{document}

\title{\bfseries\uppercase{\large Subgraphs in preferential attachment models} }
%%% REMEMBER TO CHANGE TITLE AND AUTHORS ALSO IN THE HEADER PACKAGE
\author[$ $]{Alessandro Garavaglia}
\author[$ $]{Clara Stegehuis}
\affil[$ $]{\footnotesize Department of Mathematics and
    Computer Science, Eindhoven University of Technology, 5600 MB Eindhoven, The Netherlands}
\affil[$ $]{ {\itshape Email address}: a.garavaglia@tue.nl, c.stegehuis@tue.nl\\

\vspace{0.4cm}
\textsc{keywords}: Preferential attachment, subgraphs, triangles, clustering}

\maketitle

\thispagestyle{plain}

\vspace{-1cm}
\begin{abstract}
We consider subgraph counts in general preferential attachment models with power-law degree exponent $\tau>2$. For all subgraphs $H$, we find the scaling of the expected number of subgraphs as a power of the number of vertices. We prove our results on the expected number of subgraphs by defining an optimization problem that finds the optimal subgraph structure in terms of the indices of the vertices that together span it and by using the representation of the preferential attachment model as a P\'olya urn model.
\end{abstract}

\section{Introduction}
\label{sec:intro}

The degree distribution of many real-world networks can be approximated by a power-law distribution~\cite{faloutsos1999,vazquez2002}, where for most networks the degree exponent $\tau$ was found to be between 2 and 3, so that the degree distribution has infinite variance.
Another important property of networks are subgraph counts, also referred to as motif counts. In many real-world networks, several subgraphs were found to appear more frequently than other subgraphs~\cite{milo2002}. Which type of subgraph appears most frequently varies for different networks, and the most frequently occurring subgraphs are believed to be correlated with the function of the network~\cite{milo2004,milo2002,wuchty2003}. The triangle is the most studied subgraph, allowing to compute the clustering coefficient of the network, which expresses the fraction of connected neighbors of a vertex.

To investigate which subgraphs occur more frequently than expected in a given network, the subgraph count in a given network is usually compared to the subgraph count in a random graph null model~\cite{gao2017,maugis2017,milo2004,onnela2005}. 
Several random graph models could potentially serve as null models. In practice, the null model is frequently obtained by randomly switching edges while preserving the degrees. This model however, is not mathematically tractable for $\tau<3$, so that it requires simulations to estimate the subgraph count in such networks~\cite{marcus2012,wernicke2006}. 

Several other null models for simple, scale-free networks exist, such as the configuration model~\cite{bollobas1980}, the rank-1 inhomogeneous random graph~\cite{boguna2003,chung2002}, the preferential attachment model~\cite{albert1999} or hyperbolic random graphs~\cite{krioukov2010}.
When the degree-exponent satisfies $\tau<3$, the configuration model results in a network with many multiple edges and self-loops~\cite[Chapter 7]{hofstad2009}, so that it is not a null model for simple networks anymore. A possible solution is to merge all multiple edges of the configuration model, and consider the erased configuration model instead~\cite{britton2006}. This model is mathematically tractable, and subgraph counts for this model were derived in~\cite{hofstad2017d}. 
%The rank-1 inhomogeneous random graph is closely related to the erased configuration model, and subgraph counts for this random graph null model show similar behavior as in the erased configuration model~\cite{stegehuis2018}. In both models, every subgraph typically occurs on vertices with degrees of specific orders of magnitude. 

In this paper, we analyze subgraph counts for a different random graph null model, the preferential attachment model. The preferential attachment model was first introduced by Albert and Barab\'asi~\cite{albert1999,albert2002}. In their original work, they described a growing random graph model where a new vertex appears at each time step. The new vertex connects with a fixed number of existing vertices chosen with probability proportional to the degrees.
This original Barab\'asi-Albert model has been generalized over the last years, generating the broad class of random graphs called {\em preferential attachment models} (PAMs). 

%In general, a PAM is a graph sequence $(G_t)_{t\geq 1}$ constructed recursively. At time $t+1$, a new vertex $t+1$ appears in the graph, and it is connected to $m\geq 1$ existing vertices, with probability
%\eqn{
%\label{for:PAM:heur}
%	\pr\left(t+1\rightarrow i~|~G_t\right) \propto f(i),
%} 
%where $f$ is some function of the vertex $i$ called {\em PA function}. The original Barab\'asi-Albert model is defined for $f(i) = D_i(t)$, where $D_i(t)$ is the degree of vertex $i$ at time $t$. 
The original Barabasi-Albert model is known to produce a power-law degree distribution with $\tau=3$ \cite{albert1999}. Often, a modification is considered, where edges are attached to vertices with probability proportional to the degree {\em plus a constant} $\delta$. The constant $\delta$ allows to obtain different values for the power-law exponent $\tau$. For $\delta=0$, we retrieve the original Albert-Barab\'asi model. 

In the present paper we focus on the case where $m\geq2$ is fixed, and our results hold for any value of $\delta>-m$. Taking $\delta\in(-m,0)$ results in $\tau\in(2,3)$, as observed in many real-world networks.
%In this setting, the preferential attachment model generates networks which are different from these other two null models. For example, the distances in the preferential attachment model are larger than in the configuration model \cite{caravenna2016}, even though they are of the same order of magnitude. 
An important difference between the preferential attachment model and most other random graph null models is that edges can be interpreted as directed. Thus, it allows us to study \emph{directed} subgraphs. This is a major advantage of the PAM over other random graph null models, since most real-world network subgraphs in for example biological networks are directed as well~\cite{milo2004,shen-orr2002}. 

\subsection{Literature on subgraphs in PAMs}
\label{sec:subgraphs:PAM:refs}
We now briefly summarize existing results on specific subgraph counts in preferential attachment models. The triangle is the most studied subgraph, allowing to investigate clustering in the preferential attachment model. 
Bollob\'as and Riordan \cite{bollobas2003} prove that for any integer-valued function $T(t)$ there exists a PAM with $T(t)$ triangles, where $t$ denotes the number of vertices in PAM. They further show that the clustering coefficient in the Albert-Barab\'asi model is of order $(\log t)^2/t$, while the expected number of triangles is of order $(\log t)^3$ and more generally, the expected number of cycles of length $l$ scales as $(\log t)^l$. 

Eggmann and Noble \cite{eggmann2011} consider $\delta>0$, so that $\tau>3$ and investigate the number of subgraphs for $m=1$ (so subtrees), and for $m\geq 2$ they study the number of triangles and the clustering coefficient. They observe that the expected number of triangles is of order $\log t$ while the clustering coefficient is of order $\log t/t$, which is different than the results in~\cite{bollobas2003}. Our result on general subgraphs for any value of $\delta$ in Theorem \ref{th:subgraph:expectation} explains this difference (in particular, we refer to \eqref{eq:opteq}).

In a series of papers \cite{prokhorenkova2013,prokhorenkova2016a, prokhorenkova2016} Prokhorenkova et al.\ proved results on the clustering coefficient and the number of triangles for a broad class of PAMs, assuming general properties on the attachment probabilities. These attachment probabilities are in a form that increases the probability of creating a triangle. They prove that in this setting the number of triangles is of order $t$, while the clustering coefficients behaves differently depending on the exact attachment probabilities.

\subsection{Our contribution}

For every directed subgraph, we obtain the scaling of the expected number of such subgraphs in the PAM, generalizing the above results on triangles, cycles and subtrees. Furthermore, we identify the most likely degrees of vertices participating in such subgraphs, which shows that subgraphs in the PAM are typically formed between vertices with degrees of a specific order of magnitude. The order of magnitude of these degrees can be found using an optimization problem. For general subgraphs, our results provide the scaling of the expected number of such subgraphs in the network size $t$. For the triangle subgraph, we obtain precise asymptotic results on the subgraph count, which allows to study clustering in the PAM.

We use the interpretation of the PAM as a P\'olya urn graph. This interpretation allows to view the edges as being present independently, so that we are able to obtain the probability that a subgraph $H$ is present on a specific set of vertices.

\subsection{Organization of the paper}

We first describe the specific PAM we study in Section~\ref{sec:model}. After that, we present our result on the scaling on the number of subgraphs in the PAM and the exact asymptotics on the number of triangles in Section~\ref{sec:results}. Section~\ref{sec:psubgraph} then provides an important ingredient for the proof of the scaling of the expected number of subgraphs: a lemma that describes the probability that a specific subgraph is present on a subset of vertices. After that, we prove our main results in Section~\ref{sec:proofs} and Section~\ref{sec:numberT:proof}. Finally, Section~\ref{sec:discussion} gives the conclusions and the discussion of our results.

\subsection{Model}\label{sec:model}
As mentioned in Section~\ref{sec:intro}, different versions of PAMs exist. Here we define the specific PAM we consider, which is a modification of \cite[Model 3]{berger2014}:
\begin{Definition}[Sequential PAM]
\label{def:seqPA}
Fix $m\geq 1$, $\delta>-m$. Then  $(\PA_t(m,\delta))_{t\in\N}$ is a sequence of random graphs defined as follows:
\begin{itemize}
	\item for $t=1$, $\PA_1(m,\delta)$ consists of a single vertex with no edges;
	\item for $t=2$, $\PA_2(m,\delta)$ consists of two vertices with $m$ edges between them;
	\item for $t\geq 3$, $\PA_t(m,\delta)$ is constructed recursively as follows: conditioning on the graph at time $t-1$, we add a vertex $t$ to the graph, with $m$ new edges. Edges start from vertex $t$ and, for $j=1,\ldots,m$, they are attached sequentially to vertices $E_{t,1},\ldots,E_{t,m}$ chosen with the following probability:
	\eqn{
	\label{for:defPAM}
		\pr\left(E_{t,j} = i \mid \PA_{t-1,j-1}(m,\delta)\right) = \left\{
		\begin{array}{lcr}
			\displaystyle \frac{D_i(t-1)+\delta}{2m(t-2)+ (t-1)\delta} & & \mbox{if} \quad j=1, \\
			 & & \\
			\displaystyle \frac{D_i(t-1,j-1)+\delta}{2m(t-2)+(j-1)+ (t-1)\delta} & & \mbox{if}\quad   j=2,\ldots, m.
		\end{array}\right.
	}
\end{itemize}
	In \eqref{for:defPAM}, $D_i(t-1)$ denotes the degree of $i$ in $\PA_{t-1}(m,\delta)$, while $D_i(t-1,j-1)$ denotes the degree of vertex $i$ after the first $j-1$ edges of vertex $t$ have been attached. Here we assume that $\PA_{t-1,0} = \PA_{t-1}$.
\end{Definition}
	
To keep notation light, we write $\PA_t$ instead of $\PA_t(m,\delta)$ throughout the rest of the paper. The first term in the denominator of~\eqref{for:defPAM} describes the total degree of the first $t-1$ vertices in $\PA_{t-1,j-1}$  when $t-1$ vertices are present and $j-1$ edges have been attached. The term $(t-1)\delta$ in the denominator comes from the fact that there are $t-1$ vertices to which an edge can attach. Note that we do not allow for self-loops, but we do allow for multiple edges. 

The PAM of Definition~\ref{def:seqPA} generates a random graph where the asymptotic degree sequence is close to a power law~\cite[Lemma 4.7]{hofstad2018+},
where the degree exponent $\tau$ satisfies
\begin{equation}\label{eq:tau}
\tau=3+\delta/m.
\end{equation}

\paragraph{Labeled subgraphs.}
As mentioned before, the PAM in Definition~\ref{def:seqPA} is a multigraph, i.e., any pair of vertices may be connected by $m$ different edges. One could {\em erase} multiple edges in order to obtain a simple graph, similarly to~\cite{britton2006} 
for the configuration model. In the PAM in Definition~\ref{def:seqPA} there are at most $m$ edges between any pair of vertices, so that the effect of erasing multiple edges is small, unlike the in configuration model. We do not erase edges, so that we may count a subgraph on the same set of vertices multiple times. Not erasing edges has the advantage that we do not modify the law of the graph, therefore we can directly use known results on PAM. 
\begin{figure}[tp]
%\begin{framed}
\centering
\includegraphics[width = 0.4\textwidth]{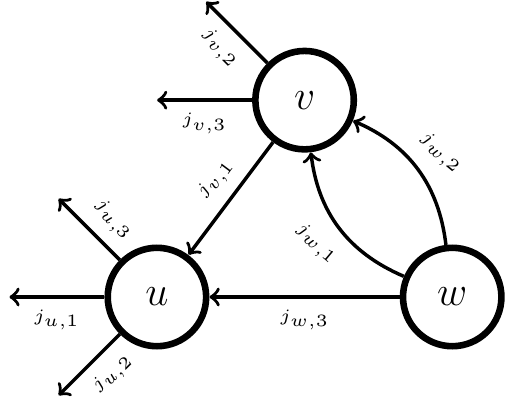}
\caption{Two labeled triangles.}
\label{fig:labelT}
%\end{framed}
\end{figure}

More precisely, to count the number of subgraphs, we analyze {\em labeled}~subgraphs, i.e., subgraphs where the edges are specified. In Figure~\ref{fig:labelT} we give the example of two labeled triangles on three vertices $u,v,w$, one consisting of edges $\{j_{v,1},j_{w,1},j_{w,3}\}$ and the other one of edges $\{j_{v,1},j_{w,2},j_{w,3}\}$.
As it turns out, the probability of two labeled subgraphs being defined by the same vertices and different edges {\em is independent of the choice of the edges}. For a more precise explanation, we refer to Section~\ref{sec:urngraph}.

\paragraph{Notation.}
We use $\plim $ for convergence in probability. We say that a sequence of events $(\mathcal{E}_n)_{n\geq 1}$ happens with high probability (w.h.p.) if $\lim_{n\to\infty}\Prob{\mathcal{E}_n}=1$. Furthermore, we write $f(n)=o(g(n))$ if $\lim_{n\to\infty}f(n)/g(n)=0$, and $f(n)=O(g(n))$ if $|f(n)|/g(n)$ is uniformly bounded, where $(g(n))_{n\geq 1}$ is nonnegative. 
We say that $X_n=O_{\sss{\prob}}(g(n))$ for a sequence of random variables $(X_n)_{n\geq 1}$ if $|X_n|/g(n)$ is a tight sequence of random variables, and $X_n=o_{\sss{\prob}}(g(n))$ if $X_n/g(n)\plim 0$. We further use the notation $[k]=\{1,2,\ldots,k\}$.

\section{Main results}\label{sec:results}
In this section, we present our results on the number of directed subgraphs in the preferential attachment model. We first define subgraphs in more detail. Let $H=(V_H,E_H)$ be a connected, directed graph. Let $\pi:V_H\mapsto 1,\ldots,|V_H|$ be a one-to-one mapping of the vertices of $H$ to $1,\ldots,|V_H|$. In the PAM, vertices arrive one by one. We let $\pi$ correspond to the order in which the vertices in $H$ have appeared in the PAM, that is $\pi(i)<\pi(j)$ if vertex $i$ was created before vertex $j$. Thus, the pair $(H,\pi)$ is a directed graph, together with a prescription of the order in which the vertices of $H$ have arrived. We call the pair $(H,\pi)$ an ordered subgraph.  

In the PAM, it is only possible for an older vertex to connect to a newer vertex but not the other way around. This puts constraints on the types of subgraphs that can be formed. We call the ordered subgraphs that can be formed in the PAM attainable. The following definition describes all attainable subgraphs:

\begin{Definition}[Attainable subgraphs]\label{def:attainable}
	Let $(H,\pi)$ be an ordered subgraph with adjacency matrix $A_\pi(H)$, where the rows and columns of the adjacency matrix are permuted by $\pi$. We say that $(H,\pi)$ is attainable if $A_\pi(H)$ defines a {\em directed acyclic graph}, where all out-degrees are less or equal than $m$.
\end{Definition}

We now investigate how many of these subgraphs are typically present in the PAM. We introduce the optimization problem
\begin{equation}\label{eq:opteq}
\begin{aligned}[b]
B(H,\pi) & = \max_{s=0,1,\dots,k} -s+\sum_{i=s+1}^k \bigg[\frac{\tau-2}{\tau-1}(\din_H(\pi^{-1}(i))-\dout_H(\pi^{-1}(i)))-\din_H(\pi^{-1}(i))\bigg] \\
& := \max_{s=0,1,\dots,k}-s+\sum_{i=s+1}^k\beta(\pi^{-1}(i)),
\end{aligned}
\end{equation}
where $\dout_H$ and $\din_H$ denote respectively the in- and the out-degree in the subgraph $H$.
Let $N_t(H,\pi)$ denote the number of times the connected graph $H$ with ordering $\pi$ occurs as a subgraph of a PAM of size $t$. 
The following theorem studies the scaling of the expected number of directed subgraphs in the PAM, and relates it to the optimization problem~\eqref{eq:opteq}:

\begin{Theorem}
	\label{th:subgraph:expectation}
	Let $H$ be a directed subgraph on $k$ vertices with ordering $\pi$ such that $(H,\pi)$ is attainable and there are $r$ different optimizers to~\eqref{eq:opteq}. Then, there exist $0<C_1\leq C_2<\infty$ such that
	\begin{equation}\label{eq:NHthm}
	C_1\leq \lim_{t\to\infty}\frac{\Exp{N_t(H,\pi)}}{t^{k+B(H,\pi)}\log^{r-1}(t)}\leq C_2.
	\end{equation}
\end{Theorem}

Theorem~\ref{th:subgraph:expectation} gives the asymptotic scaling of the number of subgraphs where the order in which the vertices appeared in the PAM is known. The total number of copies of $H$ for any ordering, $N_t(H)$, can then easily be obtained from Theorem~\ref{th:subgraph:expectation}:
\begin{Corollary}\label{cor:NH}
	Let $H$ be a directed subgraph on $k$ vertices with $\Pi\neq \varnothing$ the set of orderings $\pi$ such that $(H,\pi)$ is attainable. 
	Let 
	\begin{equation}\label{eq:opttot}
		B(H)=\max_{\pi\in\Pi}B(H,\pi),
	\end{equation}
	and let $r^*$ be the largest number of different optimizers to~\eqref{eq:opteq} among all $\pi\in\Pi$ that maximize~\eqref{eq:opttot}. Then, there exist $0<C_1\leq C_2<\infty$ such that
	\begin{equation}\label{eq:NHcol}
	C_1\leq \lim_{t\to\infty}\frac{\Exp{N_t(H)}}{t^{k+B(H)}\log^{r^*-1}(t)}\leq C_2.
	\end{equation}
\end{Corollary}

Note that from Corollary~\ref{cor:NH} it is also possible to obtain the undirected number of subgraphs in a PAM, by summing the number of all possible directed subgraphs that create some undirected subgraph when the directions of the edges are removed.

\begin{figure}[tb]
	\centering
	\definecolor{mycolor1}{RGB}{230,37,52}%
	\tikzstyle{every node}=[circle,fill=black!25,minimum size=9pt,inner sep=0pt]
	\tikzstyle{S1}=[fill=mycolor1!60]
	\tikzstyle{S2}=[fill=mycolor1!60!black]
	\tikzstyle{S3}=[fill=mycolor1]
	\tikzstyle{n1}=[fill=mycolor1!20]
	\begin{subfigure}[t]{0.2\linewidth}
		\centering
		\begin{tikzpicture}
		\tikzstyle{edge} = [draw,thick,->]
		\node[S1] (a) at (0,0) {};
		\node[S2] (b) at (0.75,0.9) {};
		\node[S3] (c) at (1.5,0) {};
		\draw[edge] (a)--(b);
		\draw[edge] (c)--(b);
		\draw[edge] (a)--(c);
		\end{tikzpicture}	
		\caption{$t^{\frac{3-\tau}{\tau-1}}\log(t)$}
		\label{fig:K3}
	\end{subfigure}
	\begin{subfigure}[t]{0.2\linewidth}
		\centering
		\begin{tikzpicture}
		\tikzstyle{edge} = [draw,thick,->]
		\node[S1] (a) at (0,0) {};
		\node[S2] (b) at (0.75,0.9) {};
		\node[S1] (c) at (1.5,0) {};
		\draw[edge] (a)--(b);
		\draw[edge] (c)--(b);
		\end{tikzpicture}	
		\caption{$t^{2/(\tau-1)}$}
		\label{fig:wedge1}
	\end{subfigure}
	\begin{subfigure}[t]{0.2\linewidth}
		\centering
		\begin{tikzpicture}
		\tikzstyle{edge} = [draw,thick,->]
		\node[S1] (a) at (0,0) {};
		\node[S1] (b) at (0.75,0.9) {};
		\node[S1] (c) at (1.5,0) {};
		\draw[edge] (a)--(b);
		\draw[edge] (b)--(c);
		\end{tikzpicture}	
		\caption{$t$}
		\label{fig:wedge2}
	\end{subfigure}
	\begin{subfigure}[t]{0.2\linewidth}
		\centering
		\begin{tikzpicture}
		\tikzstyle{edge} = [draw,thick,->]
		\node[S1] (a) at (0,0) {};
		\node[S1] (b) at (0.75,0.9) {};
		\node[S1] (c) at (1.5,0) {};
		\draw[edge] (b)--(a);
		\draw[edge] (b)--(c);
		\end{tikzpicture}	
		\caption{$t$}
		\label{fig:wedge3}
	\end{subfigure}
	\caption{Order of magnitude of $N_t(H)$ for all attainable connected directed graphs on 3 vertices and for $2<\tau<3$. Vertices with degree proportional to a constant are light pink, vertices with free degrees are bright red, and vertices of degree proportional to $t^{1/(\tau-1)}$ are dark red.}
	\label{fig:motif3}
\end{figure}
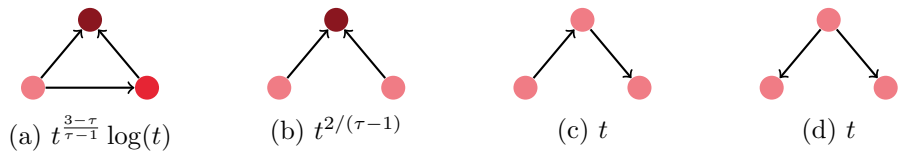

\paragraph{Interpretation of the optimization problem.}
The optimization problem~in ~\ref{eq:opteq} has an intuitive explanation. Assume that $\pi$ is the identity mapping, so that vertex 1  is the oldest vertex of $H$, vertex 2 the second oldest and so on. We show in Section~\ref{sec:proof:lem:finiteedges} that the probability that an attainable subgraph is present on vertices with indices $u_1<u_2<\cdots<u_k$ scales as
\begin{equation}
\prod_{i\in[k]}u_i^{\beta(i)},
\end{equation}
with $\beta(i)$ as in~\eqref{eq:opteq}.
Thus, if for all $i$, $u_i\propto t^{\alpha_i}$ for some $\alpha_i$, then the probability that the subgraph is present scales as $t^{\sum_{i\in[k]}\alpha_i\beta(i)}$. The number of vertices with index proportional to $t^{\alpha_i}$ scales as $t^{\alpha_i}$. Therefore, heuristically, the number of times subgraph $H$ occurs on vertices with indices proportional to $(t^{\alpha_i})_{i\in[k]}$ such that $\alpha_1\leq \alpha_2 \leq \cdots\leq \alpha_k$ scales as
\begin{equation}
 t^{\sum_{i\in[k]}(\beta(i)+1)\alpha_i}.
\end{equation}
Because the exponent is linear in $\alpha_i$, the exponent is maximized for $\alpha_i\in\{0,1\}$ for all $i$. Because of the extra constraint $\alpha_1\leq\alpha_2\leq\cdots\leq \alpha_k$ which arises from the ordering of the vertices in the PAM, the maximal value of the exponent is $k+B(H)$. This suggests that the number of subgraphs scales as $t^{k+B(H)}$. 

Thus, the optimization problem $B(H)$ finds the most likely configuration of a subgraph in terms of the indices of the vertices  involved. If the optimum is unique, the number of subgraphs is maximized by subgraphs occurring on one set of very specific vertex indices. For example, when the maximum contribution is $\alpha_i=0$, this means that vertices with constant index, the oldest vertices of the PAM are most likely to be a member of subgraph $H$ at position $i$. When $\alpha_i=1$ is the optimal contribution, vertices with index proportional to $t$, the newest vertices, are most likely to be a member of subgraph $H$ at position $i$. When the optimum is not unique, several maximizers contribute equally to the number of subgraphs, which introduces the extra logarithmic factors in~\eqref{eq:NHthm}.

\begin{figure}[tb]
	\definecolor{mycolor1}{RGB}{230,37,52}%
	\tikzstyle{every node}=[circle,fill=black!25,minimum size=8pt,inner sep=0pt]
	\tikzstyle{S1}=[fill=mycolor1!60]
	\tikzstyle{S2}=[fill=mycolor1!60!black]
	\tikzstyle{S3}=[fill=mycolor1]
	\tikzstyle{n1}=[fill=mycolor1!20]
	\tikzstyle{edge} = [draw,thick,->]
	\begin{subfigure}[t]{0.16\linewidth}
		\centering
		\begin{tikzpicture}
		\node[S1] (a) at (0,0) {};
		\node[S1] (b) at (1,0) {};
		\node[S1] (c) at (0,1) {};
		\node (d) at (1,1) {};
		\draw[edge] (a)--(b);
		\draw[edge] (c)--(b);
		\draw[edge] (b)--(d);
		\draw[edge] (a)--(c);
		\draw[edge] (a)--(d);
		\draw[edge] (c)--(d);
		\end{tikzpicture}	
		\caption{depends on $\tau$}
		\label{fig:K4}
	\end{subfigure}
	\begin{subfigure}[t]{0.16\linewidth}
		\centering
		\begin{tikzpicture}
		\node (a) at (0,0) {};
		\node (b) at (1,0) {};
		\node (c) at (0,1) {};
		\node[S2] (d) at (1,1) {};
		\draw[edge] (a)--(b);
		\draw[edge] (a)--(d);
		\draw[edge] (b)--(d);
		\draw[edge] (a)--(c);
		\draw[edge] (c)--(d);
		\end{tikzpicture}	
		\caption{depends on $\tau$}
		\label{fig:squareextra1}
	\end{subfigure}
	\begin{subfigure}[t]{0.16\linewidth}
		\centering
		\begin{tikzpicture}
		\node (a) at (0,0) {};
		\node[S2] (b) at (1,0) {};
		\node (c) at (0,1) {};
		\node[S2] (d) at (1,1) {};
		\draw[edge] (a)--(b);
		\draw[edge] (a)--(d);
		\draw[edge] (d)--(b);
		\draw[edge] (a)--(c);
		\draw[edge] (c)--(d);
		\end{tikzpicture}	
		\caption{depends on $\tau$}
		\label{fig:squareextra2}
	\end{subfigure}
	\begin{subfigure}[t]{0.16\linewidth}
		\centering
		\begin{tikzpicture}
		\node (a) at (0,0) {};
		\node[S2] (b) at (1,0) {};
		\node[S2] (c) at (0,1) {};
		\node[S2] (d) at (1,1) {};
		\draw[edge] (a)--(b);
		\draw[edge] (a)--(d);
		\draw[edge] (d)--(b);
		\draw[edge] (a)--(c);
		\draw[edge] (d)--(c);
		\end{tikzpicture}	
		\caption{depends on $\tau$}
		\label{fig:squareextra3}
	\end{subfigure}
	\begin{subfigure}[t]{0.16\linewidth}
		\centering
		\begin{tikzpicture}
		\node[S2] (a) at (0,0) {};
		\node[S2] (b) at (1,0) {};
		\node[S1] (c) at (0,1) {};
		\node[S2] (d) at (1,1) {};
		\draw[edge] (a)--(b);
		\draw[edge] (a)--(d);
		\draw[edge] (b)--(d);
		\draw[edge] (c)--(a);
		\draw[edge] (c)--(d);
		\end{tikzpicture}	
		\caption{$t^{\frac{3-\tau}{\tau-1}}$}
		\label{fig:squareextra4}
	\end{subfigure}
	\begin{subfigure}[t]{0.16\linewidth}
		\centering
		\begin{tikzpicture}
		\node[S2] (a) at (0,0) {};
		\node[S2] (b) at (1,0) {};
		\node[S1] (c) at (0,1) {};
		\node[S2] (d) at (1,1) {};
		\draw[edge] (a)--(b);
		\draw[edge] (a)--(d);
		\draw[edge] (d)--(b);
		\draw[edge] (c)--(a);
		\draw[edge] (c)--(d);
		\end{tikzpicture}	
		\caption{$t^{\frac{3-\tau}{\tau-1}}$}
		\label{fig:squareextra5}
	\end{subfigure}
	\begin{subfigure}[t]{0.16\linewidth}
		\centering
		\begin{tikzpicture}
		\node[S2] (a) at (0,0) {};
		\node[S1] (b) at (1,0) {};
		\node[S1] (c) at (0,1) {};
		\node[S2] (d) at (1,1) {};
		\draw[edge] (b)--(a);
		\draw[edge] (a)--(d);
		\draw[edge] (b)--(d);
		\draw[edge] (c)--(a);
		\draw[edge] (c)--(d);
		\end{tikzpicture}	
		\caption{$t^{\frac{6-2\tau}{\tau-1}}$}
		\label{fig:squareextra6}
	\end{subfigure}
	\begin{subfigure}[t]{0.16\linewidth}
		\centering
		\begin{tikzpicture}
		\node[S1] (a) at (0,0) {};
		\node[S2] (b) at (1,0) {};
		\node[S3] (c) at (0,1) {};
		\node[S3] (d) at (1,1) {};
		\draw[edge] (a)--(b);
		\draw[edge] (d)--(b);
		\draw[edge] (a)--(c);
		\draw[edge] (c)--(d);
		\end{tikzpicture}	
		\caption{$t^{\frac{3-\tau}{\tau-1}}\log^2(t)$}
		\label{fig:square1}
	\end{subfigure}
	\begin{subfigure}[t]{0.16\linewidth}
		\centering
		\begin{tikzpicture}
		\node[S1] (a) at (0,0) {};
		\node[S3] (b) at (1,0) {};
		\node[S3] (c) at (0,1) {};
		\node[S2] (d) at (1,1) {};
		\draw[edge] (a)--(b);
		\draw[edge] (b)--(d);
		\draw[edge] (a)--(c);
		\draw[edge] (c)--(d);
		\end{tikzpicture}	
		\caption{$t^{\frac{3-\tau}{\tau-1}}\log^2(t)$}
		\label{fig:square2}
	\end{subfigure}
	\begin{subfigure}[t]{0.16\linewidth}
		\centering
		\begin{tikzpicture}
		\tikzstyle{every node}=[circle,fill=black!25,minimum size=8pt,inner sep=0pt]
		\node[S1] (a) at (0,0) {};
		\node[S2] (b) at (1,0) {};
		\node[S2] (c) at (0,1) {};
		\node[S2] (d) at (1,1) {};
		\draw[edge] (c)--(b);
		\draw[edge] (d)--(b);
		\draw[edge] (a)--(c);
		\draw[edge] (c)--(d);
		\end{tikzpicture}	
		\caption{$t^{\frac{1}{\tau-1}}$}
		\label{fig:paw1}
	\end{subfigure}
	\begin{subfigure}[t]{0.16\linewidth}
		\centering
		\begin{tikzpicture}
		\tikzstyle{every node}=[circle,fill=black!25,minimum size=8pt,inner sep=0pt]
		\node[S1] (a) at (0,0) {};
		\node[S2] (b) at (1,0) {};
		\node[S2] (c) at (0,1) {};
		\node[S1] (d) at (1,1) {};
		\draw[edge] (c)--(b);
		\draw[edge] (d)--(b);
		\draw[edge] (a)--(c);
		\draw[edge] (d)--(c);
		\end{tikzpicture}	
		\caption{$t^{\frac{4-\tau}{\tau-1}}$}
		\label{fig:paw2}
	\end{subfigure}
	\begin{subfigure}[t]{0.16\linewidth}
		\centering
		\begin{tikzpicture}
		\tikzstyle{every node}=[circle,fill=black!25,minimum size=8pt,inner sep=0pt]
		\node[S1] (a) at (0,0) {};
		\node (b) at (1,0) {};
		\node[S2] (c) at (0,1) {};
		\node[S1] (d) at (1,1) {};
		\draw[edge] (b)--(c);
		\draw[edge] (d)--(b);
		\draw[edge] (a)--(c);
		\draw[edge] (d)--(c);
		\end{tikzpicture}	
		\caption{depends on $\tau$}
		\label{fig:paw3}
	\end{subfigure}
	\begin{subfigure}[t]{0.16\linewidth}
		\centering
		\begin{tikzpicture}
		\tikzstyle{every node}=[circle,fill=black!25,minimum size=8pt,inner sep=0pt]
		\node(a) at (0,0) {};
		\node (b) at (1,0) {};
		\node (c) at (0,1) {};
		\node (d) at (1,1) {};
		\draw[edge] (c)--(b);
		\draw[edge] (d)--(b);
		\draw[edge] (c)--(a);
		\draw[edge] (c)--(d);
		\end{tikzpicture}	
		\caption{depends on $\tau$}
		\label{fig:paw4}
	\end{subfigure}
	\begin{subfigure}[t]{0.16\linewidth}
		\centering
		\begin{tikzpicture}
		\tikzstyle{every node}=[circle,fill=black!25,minimum size=8pt,inner sep=0pt]
		\node[S2](a) at (0,0) {};
		\node[S2] (b) at (1,0) {};
		\node[S2] (c) at (0,1) {};
		\node[S1] (d) at (1,1) {};
		\draw[edge] (c)--(b);
		\draw[edge] (d)--(b);
		\draw[edge] (c)--(a);
		\draw[edge] (d)--(c);
		\end{tikzpicture}	
		\caption{$t^{\frac{3-\tau}{\tau-1}}$}
		\label{fig:paw5}
	\end{subfigure}
	\begin{subfigure}[t]{0.16\linewidth}
		\centering
		\begin{tikzpicture}
		\tikzstyle{every node}=[circle,fill=black!25,minimum size=8pt,inner sep=0pt]
		\node[S2](a) at (0,0) {};
		\node[S1] (b) at (1,0) {};
		\node[S2] (c) at (0,1) {};
		\node[S3] (d) at (1,1) {};
		\draw[edge] (b)--(c);
		\draw[edge] (d)--(b);
		\draw[edge] (c)--(a);
		\draw[edge] (d)--(c);
		\end{tikzpicture}	
		\caption{$t^{\frac{1}{\tau-1}}\log^2(t)$}
		\label{fig:paw6}
	\end{subfigure}
	\begin{subfigure}[t]{0.16\linewidth}
		\centering
		\begin{tikzpicture}
		\tikzstyle{every node}=[circle,fill=black!25,minimum size=8pt,inner sep=0pt]
		\node[S1] (a) at (0,0) {};
		\node[S1] (b) at (1,0) {};
		\node[S2] (c) at (0,1) {};
		\node[S1] (d) at (1,1) {};
		\draw[edge] (b)--(c);
		\draw[edge] (a)--(c);
		\draw[edge] (d)--(c);
		\end{tikzpicture}	
		\caption{$t^{\frac{3}{\tau-1}}$}
		\label{fig:wedge41}
	\end{subfigure}
	\begin{subfigure}[t]{0.16\linewidth}
		\centering
		\begin{tikzpicture}
		\tikzstyle{every node}=[circle,fill=black!25,minimum size=8pt,inner sep=0pt]
		\node[S1] (a) at (0,0) {};
		\node[S1] (b) at (1,0) {};
		\node[S2] (c) at (0,1) {};
		\node[S2] (d) at (1,1) {};
		\draw[edge] (b)--(c);
		\draw[edge] (a)--(c);
		\draw[edge] (c)--(d);
		\end{tikzpicture}	
		\caption{$t^{\frac{2}{\tau-1}}$}
		\label{fig:wedge42}
	\end{subfigure}
	\begin{subfigure}[t]{0.16\linewidth}
		\centering
		\begin{tikzpicture}
		\tikzstyle{every node}=[circle,fill=black!25,minimum size=8pt,inner sep=0pt]
		\node[S1] (a) at (0,0) {};
		\node[S1] (b) at (1,0) {};
		\node[S1] (c) at (0,1) {};
		\node[S1] (d) at (1,1) {};
		\draw[edge] (c)--(b);
		\draw[edge] (a)--(c);
		\draw[edge] (c)--(d);
		\end{tikzpicture}	
		\caption{$t$}
		\label{fig:wedge43}
	\end{subfigure}
	\begin{subfigure}[t]{0.16\linewidth}
		\centering
		\begin{tikzpicture}
		\tikzstyle{every node}=[circle,fill=black!25,minimum size=8pt,inner sep=0pt]
		\node[S1] (a) at (0,0) {};
		\node[S1] (b) at (1,0) {};
		\node[S1] (c) at (0,1) {};
		\node[S1] (d) at (1,1) {};
		\draw[edge] (c)--(b);
		\draw[edge] (c)--(a);
		\draw[edge] (c)--(d);
		\end{tikzpicture}	
		\caption{$t$}
		\label{fig:wedge44}
	\end{subfigure}
	\begin{subfigure}[t]{0.16\linewidth}
		\centering
		\begin{tikzpicture}
		\tikzstyle{every node}=[circle,fill=black!25,minimum size=8pt,inner sep=0pt]
		\node[S1] (a) at (0,0) {};
		\node[S1] (b) at (1,0) {};
		\node[S1] (c) at (0,1) {};
		\node[S1] (d) at (1,1) {};
		\draw[edge] (d)--(b);
		\draw[edge] (a)--(c);
		\draw[edge] (c)--(d);
		\end{tikzpicture}	
		\caption{$t$}
		\label{fig:path1}
	\end{subfigure}
	\begin{subfigure}[t]{0.16\linewidth}
		\centering
		\begin{tikzpicture}
		\tikzstyle{every node}=[circle,fill=black!25,minimum size=8pt,inner sep=0pt]
		\node[S1] (a) at (0,0) {};
		\node[S1] (b) at (1,0) {};
		\node[S2] (c) at (0,1) {};
		\node[S1] (d) at (1,1) {};
		\draw[edge] (d)--(b);
		\draw[edge] (a)--(c);
		\draw[edge] (d)--(c);
		\end{tikzpicture}	
		\caption{$t^{\frac{2}{\tau-1}}$}
		\label{fig:path2}
	\end{subfigure}
	\begin{subfigure}[t]{0.16\linewidth}
		\centering
		\begin{tikzpicture}
		\tikzstyle{every node}=[circle,fill=black!25,minimum size=8pt,inner sep=0pt]
		\node[S1] (a) at (0,0) {};
		\node[S1] (b) at (1,0) {};
		\node[S1] (c) at (0,1) {};
		\node[S1] (d) at (1,1) {};
		\draw[edge] (d)--(b);
		\draw[edge] (c)--(a);
		\draw[edge] (c)--(d);
		\end{tikzpicture}	
		\caption{$t$}
		\label{fig:path3}
	\end{subfigure}
	\begin{subfigure}[t]{0.16\linewidth}
		\centering
		\begin{tikzpicture}
		\tikzstyle{every node}=[circle,fill=black!25,minimum size=8pt,inner sep=0pt]
		\node[S1] (a) at (0,0) {};
		\node[S1] (b) at (1,0) {};
		\node[S1] (c) at (0,1) {};
		\node[S2] (d) at (1,1) {};
		\draw[edge] (b)--(d);
		\draw[edge] (a)--(c);
		\draw[edge] (c)--(d);
		\end{tikzpicture}	
		\caption{$t^{\frac{2}{\tau-1}}$}
		\label{fig:path4}
	\end{subfigure}
	\caption{Order of magnitude of $N_t(H)$ for all attainable connected directed graphs on 4 vertices and for $2<\tau<3$. Vertices with degree proportional to a constant are light pink, vertices with free degrees are bright red, and vertices of degree proportional to $t^{1/(\tau-1)}$ are dark red. Vertices where the optimizer depends on $\tau$ are gray.	}
	\label{fig:graphlet4}
\end{figure}
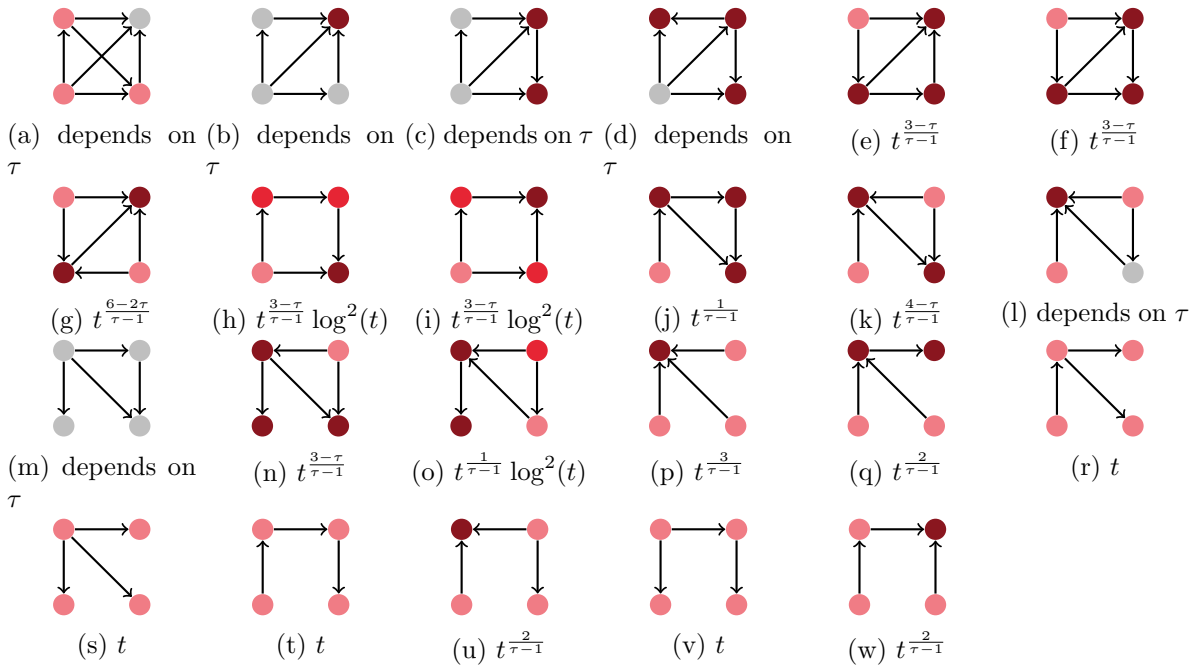

\paragraph{Most likely degrees.}
As mentioned above, the optimization problem~\eqref{eq:opteq} finds the most likely orders of magnitude of the indices of the vertices. When the optimum is unique, the optimum is attained by some vertices of constant index, and some vertices with index proportional to $t$. The vertices of constant index have degrees proportional to $t^{1/(\tau-1)}$ with high probability~\cite{hofstad2009}, whereas the vertices with index proportional to $t$ have degrees proportional to a constant. When the optimum is not unique, the indices of the vertices may have any range, so that the degrees of these vertices in the optimal subgraph structures have degrees ranging between 1 and $t^{1/(\tau-1)}$. Thus, the optimization problem~\eqref{eq:opteq} also finds the optimal subgraph structure in terms of its degrees. 
The most likely degrees of all directed connected subgraphs on 3 and 4 vertices resulting from Corollary~\ref{cor:NH} and the asymptotic number of such subgraphs for $2<\tau<3$  are visualized in Figures~\ref{fig:motif3} and~\ref{fig:graphlet4}.
For some subgraphs, the optimum of~\eqref{eq:opteq} is attained by the same $s$ and therefore the same most likely degrees for all $2<\tau<3$, while for other subgraphs the optimum may change with $\tau$. 

One such example is the complete graph of size 4. For the directed complete graph, there is only one attainable ordering satisfying Definition~\ref{def:attainable}, so we take the vertices of $H$ to be labeled with this ordering. For $\tau<5/2$, the optimizer of~\eqref{eq:opteq} is given by $s=3$ with optimal value $-3-3\tfrac{\tau-2}{\tau-1}$, whereas for $\tau>5/2$ it is given by $s=4$ and optimal value -4. Thus, for $\tau<5/2$ a complete graph of size four typically contains three hub vertices of degree proportional to $t^{1/(\tau-1)}$ and one vertex of constant degree, and the number of such subgraphs scales as $t^{1-(\tau-2)/(\tau-1)}$ whereas for $\tau>5/2$ the optimal structure contains four hub vertices instead and the number of such subgraphs scales as a constant.

\paragraph{Fluctuations of the number of subgraphs.}
In Theorem~\ref{th:subgraph:expectation} we investigate the expected number of subgraphs, which explains the average number of subgraphs over many PAM realizations. Another interesting question is what the distribution of the number of subgraphs in a PAM realization behaves like. In this paper, we mainly focus on the expected value of the number of subgraphs, but here we argue that the limiting distribution of the rescaled number of subgraphs may be quite different for different subgraphs.

In Section~\ref{sec:proof:lem:finiteedges} we show that by viewing the PAM as a P\'olya urn graph, we can associate a sequence of random independent random variables $(\psi_v)_{v\in[t]}$  to the vertices of the PAM , where $\psi_v$ has a Beta distribution with parameters depending on $m$, $\delta$ and $v$.
Once we condition on $\psi_1,\ldots,\psi_t$, the edge statuses of the graph are independent of each other. Furthermore, the degree of a vertex $v$ depends on the index $v$ and $\psi_v$. The higher $\psi_v$ is, the higher $D_v(t)$ is. Thus, we can interpret $\psi_v$ as a {\em hidden weight} associated to the vertex $v$. 

Using this representation of the PAM we can view the PAM as a random graph model with two sources of randomness: the randomness of the $\psi$-variables, and then the randomness of the independent edge statuses determined by the $\psi$-variables. 
Therefore, we can define two levels of concentration for the number of ordered subgraphs $N_t(H,\pi)$. Denote by $\E_{\sub{\psi}_t}[N_t(H,\pi)] := \E[N_t(H,\pi)\mid \psi_1,\ldots,\psi_t]$. Furthermore, let $N_{t,\psi}(H,\pi)$ denote the number of ordered subgraphs conditionally on $\psi$. 
Then, the ordered subgraph $(H,\pi)$ can be in the following three classes of subgraphs:
\begin{itemize}
	\item [$\triangleright$]
\textit{Concentrated}: $N_{t,\psi}(H,\pi)$ is concentrated around its conditional expectation $\E_{\sub{\psi}_t}[N_t(H,\pi)]$, i.e., as $t\rightarrow\infty$, 
	\eqn{\label{for:subgraph:convergence1}
		\frac{N_{t,\psi}(H,\pi)}{\E_{\sub{\psi}_t}[N_t(H,\pi)]}\stackrel{\pr}{\longrightarrow}1,
	}
and as $t\rightarrow\infty$, 
	\eqn{\label{for:subgraph:convergence2det}
		\frac{N_t(H,\pi)}{\E[N_t(H,\pi)]} \plim 1.
	} 
	\item[$\triangleright$]
	\textit{Only conditionally concentrated:} condition~\eqref{for:subgraph:convergence1} holds, and as $t\to\infty$
		\eqn{\label{for:subgraph:convergence2}
		\frac{N_t(H,\pi)}{\E[N_t(H,\pi)]} \dlim X
	} 
for some random variable $X$.
\item[$\triangleright$]
	\textit{Non-concentrated:} condition~\eqref{for:subgraph:convergence1} does not hold.
\end{itemize}
	
For example, it is easy to see that the number of subgraphs as shown in Figure~\ref{fig:wedge3} satisfies $N(H)/t\plim m(m-1)/2$, so that it is a subgraph that belongs to the class of concentrated subgraphs. Below we argue that the triangle belongs to the class of only conditionally concentrated subgraphs.  
We now give a criterion for the conditional convergence of \eqref{for:subgraph:convergence1} in the following proposition:

\begin{Proposition}[Criterion for conditional convergence]
\label{prop:criterion:conditional:subgraph}
Consider a subgraph $(H,\pi)$ such that $\E[N_t(H,\pi)]\rightarrow\infty$ as $t\rightarrow\infty$. Denote by $\hat{\mathcal{H}}$ the set of all possible subgraphs composed by two distinct copies of $(H,\pi)$ with at least one edge in common. Then, as $t\rightarrow\infty$, 
\eqn{ 
\label{for:subg:variance:1}
	\sum_{\hat{H}\in\hat{\mathcal{H}}}\E[N_t(\hat{H})] = o\Big(\E[N_t(H,\pi)]^2\Big) \qquad \Longrightarrow \qquad
			\frac{N_{t,\psi}(H,\pi)}{\E_{\sub{\psi}_t}[N_t(H,\pi)]}\plim1.
}
\end{Proposition}
Proposition~\ref{prop:criterion:conditional:subgraph} gives a simple criterion for conditional convergence for a subgraph $(H,\pi)$, and it is proved in Section~\ref{sec:fluct}. 
The condition in \eqref{for:subg:variance:1} is simple to evaluate in practice. 
We denote the subgraphs consisting of two overlapping copies of $(H,\pi)$ sharing at least one edge by $\hat{H}_1,\dots,\hat{H}_r$. To identify the order of magnitude of $\E[\hat{H}_i]$, we apply Corollary~\ref{cor:NH} to $\hat{H}_i$ or, in other words, we apply Theorem~\ref{th:subgraph:expectation} to all possible orderings $\hat{\pi}$ of $\hat{H}_i$. Once we have all orders of magnitude of $(\hat{H}_i,\hat{\pi})$ for all orderings $\hat{\pi}$, and for all $\hat{H}_i$, it is immediate to see if hypothesis of Proposition~\ref{prop:criterion:conditional:subgraph} is satisfied. 

There are subgraphs where the condition in Proposition~\ref{prop:criterion:conditional:subgraph} does not hold. For example, merging two copies of the subgraph of Figure~\ref{fig:wedge42} as in Figure~\ref{fig:wedgemerged} violates the condition in Proposition~\ref{prop:criterion:conditional:subgraph}. We show in Section~\ref{sec:fluct} that this subgraph is in the class of non-concentrated subgraphs with probability close to one.

	\begin{figure}[t]
	\centering
		\definecolor{mycolor1}{RGB}{230,37,52}%
	\tikzstyle{every node}=[circle,fill=black!25,minimum size=8pt,inner sep=0pt]
	\tikzstyle{S1}=[fill=mycolor1!60]
	\tikzstyle{S2}=[fill=mycolor1!60!black]
	\tikzstyle{S3}=[fill=mycolor1]
	\tikzstyle{n1}=[fill=mycolor1!20]
	\tikzstyle{edge} = [draw,thick,->]
	\begin{tikzpicture}
	\tikzstyle{every node}=[circle,fill=black!25,minimum size=8pt,inner sep=0pt]
	\node[S1] (a) at (0,0) {};
	\node[S1] (b) at (0,1) {};
	\node[S2] (c) at (1,0) {};
	\node[S2] (d) at (1,1) {};
	\node[S1] (e) at (2,0) {};
	\node[S1] (f) at (2,1) {};
	\draw[edge] (b)--(c);
	\draw[edge] (a)--(c);
	\draw[edge] (c)--(d);
	\draw[edge] (f)--(c);
	\draw[edge] (e)--(c);
	\end{tikzpicture}	
	\caption{The order of magnitude of this subgraph containing two merged copies of the subgraph of Figure~\ref{fig:wedge42} is $t^{\frac{4}{\tau-1}}$, so that the condition in Proposition~\ref{prop:criterion:conditional:subgraph} is not satisfied for the subgraph in Figure~\ref{fig:wedge42}.}
	\label{fig:wedgemerged}
\end{figure}
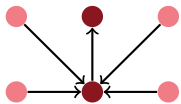

\subsection{Exact constants: triangles}
Theorem \ref{th:subgraph:expectation} allows to identify the order of magnitude of the expected number of subgraphs in PAM. In particular, for a subgraph $H$ with ordering $\pi$, it assures the existence of two constants $0<C_1\leq C_2<\infty$ as in \eqref{eq:NHthm}. 
A more detailed analysis is necessary to prove a stronger result than Theorem \ref{th:subgraph:expectation} of the type
$$
	\lim_{t\rightarrow\infty} \frac{\Exp{N_t(H,\pi)}}{t^{k+B(H,\pi)}\log^{r-1}(t)}= C,
$$
for some constant $0<C<\infty$. In other words, given an ordered subgraph $(H,\pi)$, we want to identify the constant $C>0$ such that
\eqn{
\label{for:expect:subgraph:constant}
	\Exp{N_t(H,\pi)} = Ct^{k+B(H,\pi)}\log^{r-1}(t)(1+o(1)).
}
We prove \eqref{for:expect:subgraph:constant} for triangles to show the difficulties in the evaluation of the precise constant $C$ for general subgraphs.
The following theorem provides the detailed scaling of the expected number of triangles:

\begin{Theorem}[Phase transition for the number of triangles]
\label{th:numberT}
Let $m\geq 2$ and $\delta>-m$ be parameters for $(\PA_t)_{t\geq1}$.  Denote the number of labeled triangles in $\PA_t$ by $\triangle_t$. Then, as $t\rightarrow\infty$, 
\begin{enumerate}
	\item if $\tau>3$, then
	$$
		\E[\triangle_t] = \frac{m^2(m-1)(m+\delta)(m+\delta+1)}{\delta^2(2m+\delta)}\log (t)(1+o(1));
	$$
	\item if $\tau=3$, then 
	$$
		 \E[\triangle_t] = \frac{m(m-1)(m+1)}{48}\log^3 (t)(1+o(1));
	$$
	\item if  $\tau\in(2,3)$, then
	$$
		\E[\triangle_t]=\frac{m^2(m-1)(m+\delta)(m+\delta+1)}{\delta^2(2m+\delta)}t^{(3-\tau)/(\tau-1)}\log (t)(1+o(1)).
	$$
\end{enumerate}
\end{Theorem}
Theorem~ \ref{th:numberT} in the case $\delta=0$ coincides with~ \cite[Theorem 14]{bollobas2003}. For $\delta>0$ we retrieve the result in~ \cite[Proposition 4.3]{eggmann2011}, noticing that the additive constant $\beta$ in the attachment probabilities in the M\'ori model considered in~ \cite{eggmann2011} coincides with~\eqref{def:seqPA} for $\beta = \delta/m$. 

The proof of Theorem~\ref{th:numberT} in Section~\ref{sec:numberT:proof} shows that to identify the constant in~ \eqref{for:expect:subgraph:constant} we need to evaluate the precise expectations involving the attachment probabilities of edges. The equivalent formulation of PAM given in Section~\ref{sec:urngraph} simplifies the calculations, but it is still necessary to evaluate rather complicated expectations involving  products of several terms as in~\eqref{for:edgeUrnP}. For a more detailed discussion, we refer to Remark~\ref{rem:subgraphs:constant}.

\paragraph{The distribution of the number of triangles.}
Theorem~\ref{th:numberT} shows the behavior of the expected number of triangles. The distribution of the number of triangles across various PAM realizations is another object of interest. We prove the following result for the number of triangles $\triangle_t$:

\begin{Corollary}[Conditional concentration of triangles]
\label{cor:conc:triang}
For $\tau\in(2,3)$, the number of triangles $\triangle_t$ is conditionally concentrated in the sense of \eqref{for:subgraph:convergence1}.
\end{Corollary}
Corollary~\ref{cor:conc:triang} is a direct consequence of Proposition~\ref{prop:criterion:conditional:subgraph}, and the atlas of the order of magnitudes of all possible realizations of the subgraphs consisting of two triangles sharing one or two edges, presented in Figure~\ref{fig:graphlet4tr}.
Figure~\ref{fig:triadfluct} shows a density approximation of the number of triangles obtained by simulations. These figures suggest that the rescaled number of triangles converges to a random limit, since the width of the density plots does not decrease in $t$. Thus, while the number of triangles concentrates conditionally, it does not seem to converge to a constant when taking the random $\psi$-variables into account. This would put the triangle subgraph in the class of only conditionally concentrated subgraphs. Proving this and identifying the limiting random variable of the number of triangles is an interesting open question.

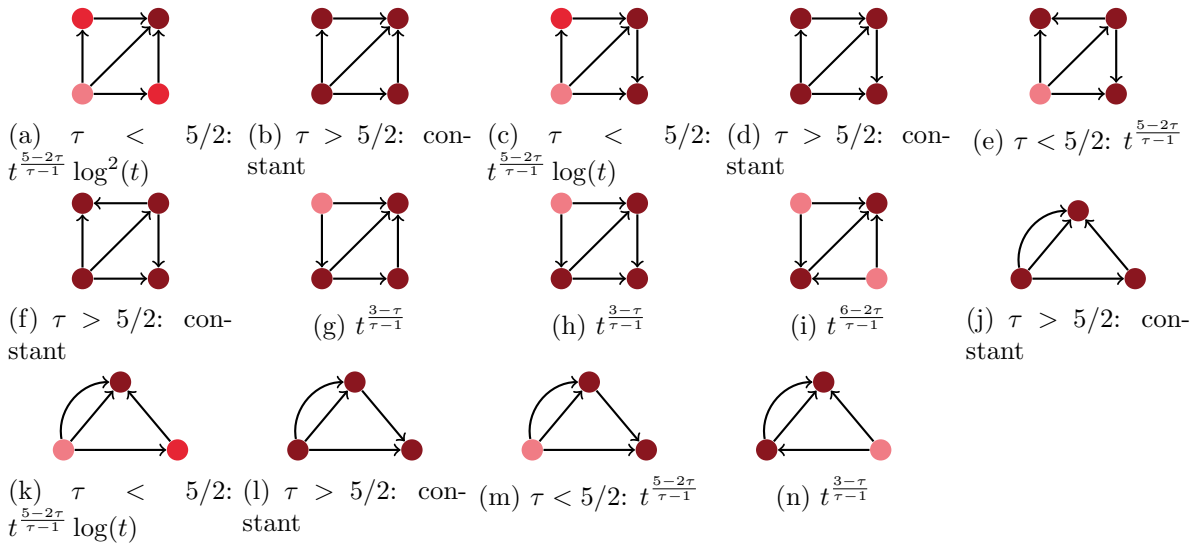
\begin{figure}[tp!]
	\definecolor{mycolor1}{RGB}{230,37,52}%
	\tikzstyle{every node}=[circle,fill=black!25,minimum size=8pt,inner sep=0pt]
	\tikzstyle{S1}=[fill=mycolor1!60]
	\tikzstyle{S2}=[fill=mycolor1!60!black]
	\tikzstyle{S3}=[fill=mycolor1]
	\tikzstyle{n1}=[fill=mycolor1!20]
	\tikzstyle{edge} = [draw,thick,->]
	\begin{subfigure}[t]{0.19\linewidth}
		\centering
		\begin{tikzpicture}
		\node[S1] (a) at (0,0) {};
		\node[S3] (b) at (1,0) {};
		\node[S3] (c) at (0,1) {};
		\node[S2] (d) at (1,1) {};
		\draw[edge] (a)--(b);
		\draw[edge] (a)--(d);
		\draw[edge] (b)--(d);
		\draw[edge] (a)--(c);
		\draw[edge] (c)--(d);
		\end{tikzpicture}	
		\caption{$\tau<5/2$: $t^{\frac{5-2\tau}{\tau-1}}\log^2(t)$}
		\label{fig:squareextra1trsmall}
	\end{subfigure}
	\begin{subfigure}[t]{0.19\linewidth}
	\centering
	\begin{tikzpicture}
	\node[S2] (a) at (0,0) {};
	\node[S2] (b) at (1,0) {};
	\node[S2] (c) at (0,1) {};
	\node[S2] (d) at (1,1) {};
	\draw[edge] (a)--(b);
	\draw[edge] (a)--(d);
	\draw[edge] (b)--(d);
	\draw[edge] (a)--(c);
	\draw[edge] (c)--(d);
	\end{tikzpicture}	
	\caption{$\tau>5/2$: constant}
	\label{fig:squareextra1trlarge}
\end{subfigure}
	\begin{subfigure}[t]{0.19\linewidth}
		\centering
		\begin{tikzpicture}
		\node[S1] (a) at (0,0) {};
		\node[S2] (b) at (1,0) {};
		\node[S3] (c) at (0,1) {};
		\node[S2] (d) at (1,1) {};
		\draw[edge] (a)--(b);
		\draw[edge] (a)--(d);
		\draw[edge] (d)--(b);
		\draw[edge] (a)--(c);
		\draw[edge] (c)--(d);
		\end{tikzpicture}	
		\caption{$\tau<5/2$: $t^{\frac{5-2\tau}{\tau-1}}\log(t)$}
		\label{fig:squareextra2trsmall}
	\end{subfigure}
	\begin{subfigure}[t]{0.19\linewidth}
	\centering
	\begin{tikzpicture}
	\node[S2] (a) at (0,0) {};
	\node[S2] (b) at (1,0) {};
	\node[S2] (c) at (0,1) {};
	\node[S2] (d) at (1,1) {};
	\draw[edge] (a)--(b);
	\draw[edge] (a)--(d);
	\draw[edge] (d)--(b);
	\draw[edge] (a)--(c);
	\draw[edge] (c)--(d);
	\end{tikzpicture}	
	\caption{$\tau>5/2$: constant}
	\label{fig:squareextra2tr}
\end{subfigure}
	\begin{subfigure}[t]{0.19\linewidth}
		\centering
		\begin{tikzpicture}
		\node[S1] (a) at (0,0) {};
		\node[S2] (b) at (1,0) {};
		\node[S2] (c) at (0,1) {};
		\node[S2] (d) at (1,1) {};
		\draw[edge] (a)--(b);
		\draw[edge] (a)--(d);
		\draw[edge] (d)--(b);
		\draw[edge] (a)--(c);
		\draw[edge] (d)--(c);
		\end{tikzpicture}	
		\caption{$\tau<5/2$: $t^{\frac{5-2\tau}{\tau-1}}$}
		\label{fig:squareextra3trsmall}
	\end{subfigure}
	\begin{subfigure}[t]{0.19\linewidth}
	\centering
	\begin{tikzpicture}
	\node[S2] (a) at (0,0) {};
	\node[S2] (b) at (1,0) {};
	\node[S2] (c) at (0,1) {};
	\node[S2] (d) at (1,1) {};
	\draw[edge] (a)--(b);
	\draw[edge] (a)--(d);
	\draw[edge] (d)--(b);
	\draw[edge] (a)--(c);
	\draw[edge] (d)--(c);
	\end{tikzpicture}	
	\caption{$\tau>5/2$: constant}
	\label{fig:squareextra3tr}
\end{subfigure}
	\begin{subfigure}[t]{0.19\linewidth}
		\centering
		\begin{tikzpicture}
		\node[S2] (a) at (0,0) {};
		\node[S2] (b) at (1,0) {};
		\node[S1] (c) at (0,1) {};
		\node[S2] (d) at (1,1) {};
		\draw[edge] (a)--(b);
		\draw[edge] (a)--(d);
		\draw[edge] (b)--(d);
		\draw[edge] (c)--(a);
		\draw[edge] (c)--(d);
		\end{tikzpicture}	
		\caption{$t^{\frac{3-\tau}{\tau-1}}$}
		\label{fig:squareextra4tr}
	\end{subfigure}
	\begin{subfigure}[t]{0.19\linewidth}
		\centering
		\begin{tikzpicture}
		\node[S2] (a) at (0,0) {};
		\node[S2] (b) at (1,0) {};
		\node[S1] (c) at (0,1) {};
		\node[S2] (d) at (1,1) {};
		\draw[edge] (a)--(b);
		\draw[edge] (a)--(d);
		\draw[edge] (d)--(b);
		\draw[edge] (c)--(a);
		\draw[edge] (c)--(d);
		\end{tikzpicture}	
		\caption{$t^{\frac{3-\tau}{\tau-1}}$}
		\label{fig:squareextra5tr}
	\end{subfigure}
	\begin{subfigure}[t]{0.19\linewidth}
		\centering
		\begin{tikzpicture}
		\node[S2] (a) at (0,0) {};
		\node[S1] (b) at (1,0) {};
		\node[S1] (c) at (0,1) {};
		\node[S2] (d) at (1,1) {};
		\draw[edge] (b)--(a);
		\draw[edge] (a)--(d);
		\draw[edge] (b)--(d);
		\draw[edge] (c)--(a);
		\draw[edge] (c)--(d);
		\end{tikzpicture}	
		\caption{$t^{\frac{6-2\tau}{\tau-1}}$}
		\label{fig:squareextra6tr}
	\end{subfigure}
	\begin{subfigure}[t]{0.19\linewidth}
		\centering
		\begin{tikzpicture}
		\tikzstyle{edge} = [draw,thick,->]
		\node[S2] (a) at (0,0) {};
		\node[S2] (b) at (0.75,0.9) {};
		\node[S2] (c) at (1.5,0) {};
		\draw[edge] (a)--(b);
		\draw [thick,->,bend left = 50] (a) to (b);
		\draw[edge] (c)--(b);
		\draw[edge] (a)--(c);
		\end{tikzpicture}	
		\caption{$\tau>5/2$: constant}
		\label{fig:squareextra7tr}
	\end{subfigure}
	\begin{subfigure}[t]{0.19\linewidth}
		\centering
		\begin{tikzpicture}
		\tikzstyle{edge} = [draw,thick,->]
		\node[S1] (a) at (0,0) {};
		\node[S2] (b) at (0.75,0.9) {};
		\node[S3] (c) at (1.5,0) {};
		\draw[edge] (a)--(b);
		\draw [thick,->,bend left = 50] (a) to (b);
		\draw[edge] (c)--(b);
		\draw[edge] (a)--(c);
		\end{tikzpicture}	
		\caption{$\tau<5/2$: $t^{\frac{5-2\tau}{\tau-1}}\log(t)$}
		\label{fig:squareextra8tr}
	\end{subfigure}
	\begin{subfigure}[t]{0.19\linewidth}
		\centering
		\begin{tikzpicture}
		\tikzstyle{edge} = [draw,thick,->]
		\node[S2] (a) at (0,0) {};
		\node[S2] (b) at (0.75,0.9) {};
		\node[S2] (c) at (1.5,0) {};
		\draw[edge] (a)--(b);
		\draw [thick,->,bend left = 50] (a) to (b);
		\draw[edge] (b)--(c);
		\draw[edge] (a)--(c);
		\end{tikzpicture}	
		\caption{$\tau>5/2$: constant}
		\label{fig:squareextra9tr}
	\end{subfigure}
	\begin{subfigure}[t]{0.19\linewidth}
		\centering
		\begin{tikzpicture}
		\tikzstyle{edge} = [draw,thick,->]
		\node[S1] (a) at (0,0) {};
		\node[S2] (b) at (0.75,0.9) {};
		\node[S2] (c) at (1.5,0) {};
		\draw[edge] (a)--(b);
		\draw [thick,->,bend left = 50] (a) to (b);
		\draw[edge] (b)--(c);
		\draw[edge] (a)--(c);
		\end{tikzpicture}	
		\caption{$\tau<5/2$: $t^{\frac{5-2\tau}{\tau-1}}$}
		\label{fig:squareextra10tr}
	\end{subfigure}
	\begin{subfigure}[t]{0.19\linewidth}
		\centering
		\begin{tikzpicture}
		\tikzstyle{edge} = [draw,thick,->]
		\node[S2] (a) at (0,0) {};
		\node[S2] (b) at (0.75,0.9) {};
		\node[S1] (c) at (1.5,0) {};
		\draw[edge] (a)--(b);
		\draw [thick,->,bend left = 50] (a) to (b);
		\draw[edge] (c)--(b);
		\draw[edge] (c)--(a);
		\end{tikzpicture}	
		\caption{$t^{\frac{3-\tau}{\tau-1}}$}
		\label{fig:squareextra11tr}
	\end{subfigure}
	\caption{Order of magnitude of $N_t(H)$ for all merged triangles on 4 vertices and for $2<\tau<3$. Vertices with degree proportional to a constant are light pink, vertices with free degrees are bright red, and vertices of degree proportional to $t^{1/(\tau-1)}$ are dark red. Vertices where the optimizer depends on $\tau$ are gray.	}
	\label{fig:graphlet4tr}
\end{figure}
\begin{figure}[tb]
	\centering
	\begin{subfigure}{0.32\linewidth}
		\centering
		\includegraphics[width=\textwidth]{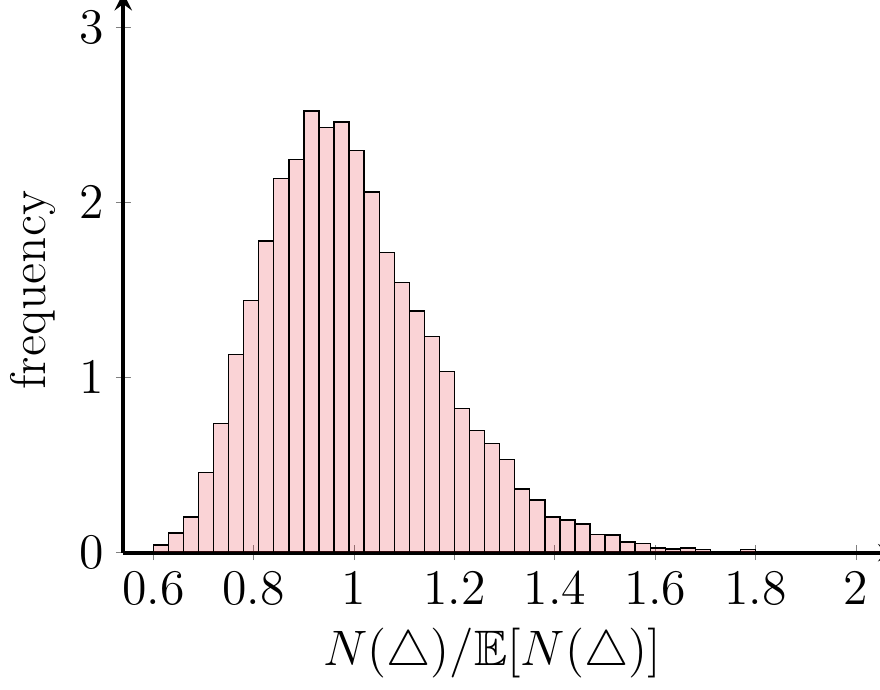}
		\caption{$t=10^5$}
		\label{fig:tr25100000}
	\end{subfigure}
	\begin{subfigure}{0.32\linewidth}
		\centering
		\includegraphics[width=\textwidth]{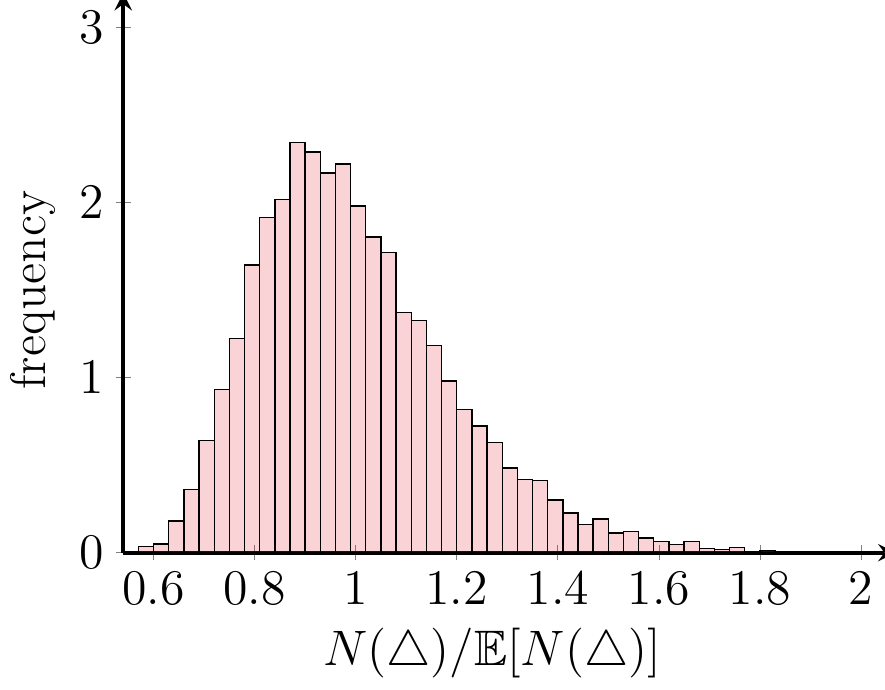}
		\caption{$t=10^6$}
		\label{fig:tr251000000}
	\end{subfigure}
	\begin{subfigure}{0.32\linewidth}
		\centering
		\includegraphics[width=\textwidth]{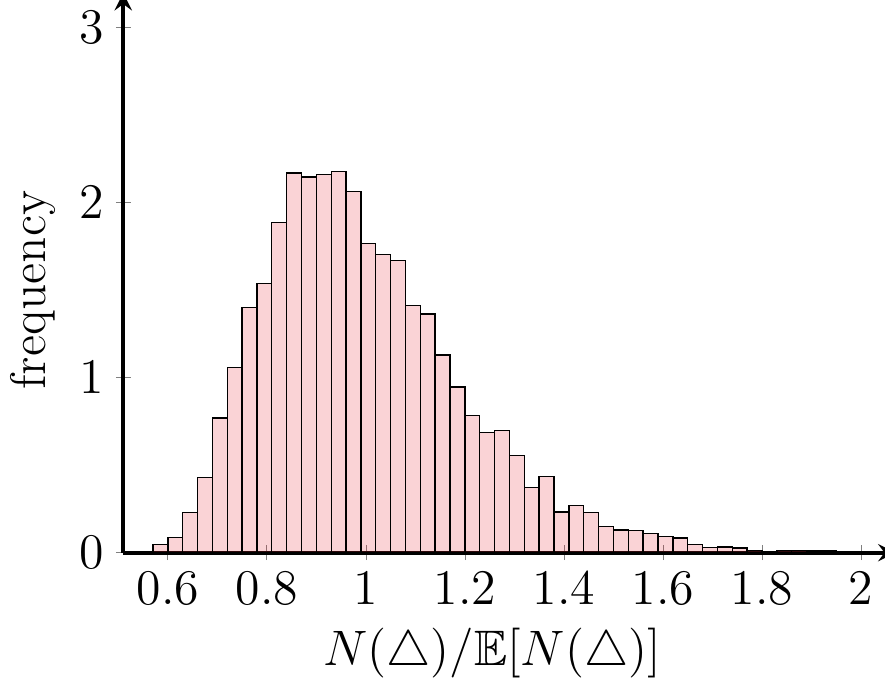}
		\caption{$t=5\cdot 10^6$}
		\label{fig:tr255000000}
	\end{subfigure}
	\caption{Density approximation of the number of triangles in $10^4$ realizations of the preferential attachment model with $\tau=2.5$ and various values of $t$.}
	\label{fig:triadfluct}
\end{figure}

\section{The probability of a subgraph being present}\label{sec:psubgraph}
In this section, we prove the main ingredient for the proof of Theorem~\ref{th:subgraph:expectation}, the probability of a subgraph being present on a given set of vertices. 
The most difficult part of evaluating the probability of a subgraph $H$ being present in $\PA_t$ is that the PAM is constructed recursively.
We consider triangles as an example. We write the event of a {\em labeled} triangle being present by $\{u\stackrel{j_1}{\leftarrow} v,\ u\stackrel{j_2}{\leftarrow}w, \ v\stackrel{j_3}{\leftarrow} w\}$, where $\{u\stackrel{j}{\leftarrow}v\}$ denotes the event that the $j$-th edge of vertex $v$ is attached to vertex $u$. Notice that in this way we express precisely which edges we consider in the triangle construction. Then,
\eqn{
\label{for:probtriangle}
\begin{split}
&\pr\left(u\stackrel{j_1}{\leftarrow} v,\ u\stackrel{j_2}{\leftarrow}w, \ v\stackrel{j_3}{\leftarrow} w\right) \\
	&= \E\bigg[\pr\left(u\stackrel{j_1}{\leftarrow} v,\ u\stackrel{j_2}{\leftarrow}w, \ v\stackrel{j_3}{\leftarrow} w \ | \ \PA_{t-1, j_3-1}\right) \bigg]
	\\ &= \E\bigg[\I\{u\stackrel{j_1}{\leftarrow} v,\ u\stackrel{j_2}{\leftarrow}w\}\frac{D_v(w-1,j_3-1)+\delta}{2m(w-2)+(j_3-1)+ (w-1)\delta}\bigg].
\end{split}
}
In \eqref{for:probtriangle}, the indicator function $\I\{u\stackrel{j_1}{\leftarrow} v,\ u\stackrel{j_2}{\leftarrow}w\}$ and $D_v(w-1,j_3-1)$ are not independent, therefore evaluating the expectation on the right-hand side of~ \eqref{for:probtriangle} is not easy. A possible solution for the evaluation of the expectation in~ \eqref{for:probtriangle} is to rescale $D_v(w-1,j_3-1)$ with an appropriate constant to obtain a martingale, and then recursively use the conditional expectation. For a detailed explanation of this, we refer to~ \cite{Bol01,Szym05} and \cite[Section 8.3]{hofstad2009}. This method is hardly tractable due to the complexity of the constants appearing (see Remark~ \ref{rem:subgraphs:constant} for a more detailed explanation).

We use a different approach to evaluate of the expectation in ~\eqref{for:probtriangle} using the interpretation of the PAM as a P\'olya urn graph, focusing mainly on the {\em the age (the indices) of the vertices}, and not on precise constants.
We give a lower and upper bound of the probability of having a finite number of edges present in the graph, as formulated in the following lemma:

\begin{Lemma}[Probability of finite set of labeled edges]
\label{lem:finiteedges}
Fix $\ell\in\N$. For vertices $\sub{u}_\ell = (u_1,\ldots,u_\ell)\in[t]^\ell$ and $\sub{v}_\ell = (v_1,\ldots,v_\ell)\in[t]^\ell$ and edge labels $\sub{j}_\ell = (j_1,\ldots,j_\ell)\in [m]^\ell$, consider the corresponding finite set of $\ell$ distinct labeled edges $M_\ell(\sub{u}_\ell,\sub{v}_\ell,\sub{j}_\ell)$. Assume that the subgraph defined by set $M_\ell(\sub{u}_\ell,\sub{v}_\ell,\sub{j}_\ell)$ is attainable in the sense of Definition ~\ref{def:attainable}. Define $\chi=(m+\delta)/(2m+\delta)$. Then:
\begin{enumerate}
\item There exist two constants $c_1(m,\delta, \ell),c_2(m,\delta, \ell)>0$ such that
\eqn{
\label{for:lem:finiteedges}
	c_1(m,\delta, \ell)\prod_{l=1}^{\ell}u_l^{\chi-1}v_l^{-\chi}
		  \leq \ \pr\left(M_\ell(\sub{u}_\ell,\sub{v}_\ell,\sub{j}_\ell) \subseteq E(\PA_t)\right) \
	 \leq c_2(m,\delta, \ell)\prod_{l=1}^{\ell}u_l^{\chi-1}v_l^{-\chi}.
}
\item  Define the set 
\eqn{\label{for:lem:finiteedges:2}
	J(\sub{u}_\ell, \sub{v}_\ell) = \left\{\sub{j}_\ell\in[m]^\ell \colon M_\ell(\sub{u}_\ell,\sub{v}_\ell,\sub{j}_\ell) \subseteq E(\PA_t)\right\}.
}
Then, there exist two constants $\hat{c}_1(m,\delta,\ell), \hat{c}_2(m,\delta,\ell)>0$ such that 
\eqn{\label{for:lem:finiteedges:3}
	\hat{c}_1(m,\delta,\ell)\prod_{l=1}^{\ell}u_l^{\chi-1}v_l^{-\chi} \leq \E[|J(\sub{u}_\ell, \sub{v}_\ell)|] 
		\leq \hat{c}_2(m,\delta,\ell)\prod_{l=1}^{\ell}u_l^{\chi-1}v_l^{-\chi} .
}
\end{enumerate}
\end{Lemma}
Formula~\eqref{for:lem:finiteedges} in the above lemma bounds the probability that a subgraph is present on vertices $\boldsymbol{u}_\ell$ and $\boldsymbol{v}_\ell$ such that the $j_i$-th edge from $u_i$ connects to $v_i$. Notice that~\eqref{for:lem:finiteedges} is independent of the precise edge labels $(j_1,\ldots,j_\ell)$. To be able to count all subgraphs, and not only subgraphs where the edge labels have been specified,~\eqref{for:lem:finiteedges:3} bounds the expected number of times a specific subgraph is present on vertices $\boldsymbol{u}_\ell$ and $\boldsymbol{v}_\ell$. This number is given exactly by the elements in set $J(\sub{u}_\ell,\sub{v}_\ell)$ as in~\eqref{for:lem:finiteedges:2}. Note that the expectation in~\eqref{for:lem:finiteedges:3} may be larger than one, due to the fact that the PAM is a multigraph. 

Lemma~\ref{lem:finiteedges} gives a bound on the probability of presence of $\ell\in\N$ distinct edges in the graph as function of the indices $(u_1,v_1),\ldots,(u_\ell,v_\ell)$ of the endpoints of the $\ell$ edges. 
Due to the properties of PAM, the index of a vertex is an indicator of its degree, due to the old-get-richer effect.  Lemma~\ref{lem:finiteedges} is a stronger result than \cite[Corollary 2.3]{DSvdH}, which gives an upper bound of the form in~ \eqref{for:lem:finiteedges} only for self-avoiding paths.

%\begin{Remark}[Factorizable events]
%{\em 
%The problem of evaluating the probability of a subgraph being present in PAM has been partially solved in \cite[Definition 4.4]{caravenna2016}, with the definition of {\em factorizable events}. In simple words, probabilities of the type in \eqref{for:probtriangle} can be written by conditioning on all the history of the graph, i.e., specifying the degree evolution of vertices contained in the structure we want to investigate. Factorizable events turned out to be useful for particular objects, such as trees where vertices have minimum degree. 
%}
%\end{Remark}

The proof of Lemma~\ref{lem:finiteedges} is based on the interpretation of the PAM in Definition~\ref{def:seqPA} as a urn experiment as proposed in \cite{berger2014}. We now introduce urn schemes and state the preliminary results we need for the proof of Lemma~ \ref{lem:finiteedges}, which is given in Section~\ref{sec:proof:lem:finiteedges}.

\subsection{P\'olya urn graph}
\label{sec:urngraph}
An urn scheme consists of an urn, with blue balls and red balls. At every time step, we draw a ball from the urn and we replace it by two balls of the same color. We start with $B_0 = b_0$ blue balls and $R_0 = r_0$ red balls. We consider two weight functions
\eqn{
\label{for:urnscheme}
	W_b(k) = a_b + k,\quad \mbox{ and } \quad W_r(k) = a_r+k.
} 
Conditionally on the number of blue balls $B_n$ and red balls $R_n$, at time $n+1$ the probability of drawing a blue ball is equal to
$$
	\frac{W_b(B_n)}{W_b(B_n)+W_r(R_n)}.
$$
The evolution of the number of balls $((B_n,R_n))_{n\in\N}$ obeys~\cite[Theorem 4.2]{hofstad2018+}
\eqn{
	\pr\left(B_n = B_0 + k\right) = \E\left[\pr\left(\left. \mathrm{Bin}(n,\psi)=k\right|\psi\right)\right],
}
where $\psi$ has a Beta distribution with parameters $B_0+a_b$ and $R_0+a_r$. In other words, the number of blue balls (equivalently, of red balls) is given by a Binomial distribution with a random probability of success $\psi$ (equivalently, $1-\psi$). 
Sometimes we call the random variable $\psi$ the {\em intensity} or {\em strength} of the blue balls in the urn. We can also see the urn process as two different urns, one containing only blue balls and the other only red balls, and we choose a urn proportionally to the number of balls in the urns. In this case, the result is the same, but we can say that $\psi$ is the strength of the blue balls urn and $1-\psi$ is the strength of the red balls urn.

The sequential model $\PA_t$ can be interpreted as experiment with $t$ urns, where the number of balls  in each urn represent the degree of a vertex in the graph. First, we introduce a random graph model:
\begin{Definition}[P\'olya urn graph]
\label{def:urngraph}
Fix $m\geq 1$ and $\delta>-m$. Let $t\in\N$ be the size of the graph. Let $\psi_1 = 1$, and consider $\psi_2,\ldots,\psi_t$ independent random variables, where 
\eqn{
\label{for:psik}
	\psi_k \stackrel{d}{=} \mathrm{Beta}\left(m+\delta, m(2k-3)+(k-1)\delta\right).
} 
Define
\eqn{
\label{for:intervalsPU}
	\varphi_j = \psi_j\prod_{i=j+1}^t (1-\psi_i), \quad S_k = \sum_{j=1}^k \varphi_j, \quad I_k = [S_{k-1},S_k).
}
Conditioning on $\psi_1,\ldots,\psi_t$, let $\{U_{k,j}\}_{k=2,\ldots,t}^{j=1,\ldots,m}$ be independent random variables, with $U_{k,j}$ uniformly distributed on $[0,S_{k-1}]$. Then, the  corresponding P\'olya urn graph $\PU_t$ is the graph of size $t$ where, for $u<v$, the number of edges between $u$ and $v$ is equal to the number of variables $U_{v,j}$ in $I_u$, for $j=1,\ldots,m$ (multiple edges are allowed).
\end{Definition}

The two sequences of graphs $(\PA_t)_{t\in\N}$ and $(\PU_t)_{t\in\N}$ have the same distribution~\cite[Theorem 2.1]{berger2014}, \cite[Chapter 4]{hofstad2018+}.
%This approach turns out to be useful in studying the local weak limit of PAM as in Definition \ref{def:seqPA} in the Benjamini-Schram sense (see \cite{berger2014}). 
The Beta distributions in Definition \ref{def:urngraph} come from the P\'olya urn interpretation of the sequential model, using urns with affine weight functions.

The formulation in Definition \ref{def:urngraph} in terms of urn experiments allows us to investigate the presence of subgraphs in an easier way than with the formulation given in Definition \ref{def:seqPA} since the dependent random variables in~\eqref{for:probtriangle}, are replaced by the product of independent random variables. We now state two lemmas that are the main ingredients for proving Lemma~\ref{lem:finiteedges}:

\begin{Lemma}[Attachment probabilities]
\label{lem:attach:urngraph}
	Consider $\mathrm{PU}_t$ as in Definition \ref{def:urngraph}. Then,
	\begin{enumerate}
		\item for $k\in[t]$, 
		\eqn{
		\label{for:SkUrnP}
			S_k = \prod_{h=k+1}^t (1-\psi_h);
		}
		\item conditioning on $\psi_1,\ldots,\psi_t$, the probability that the $j$-th edge of $k$ is attached to $v$ is equal to 
	\eqn{
	\label{for:edgeUrnP}
		\pr\left(U_{k,j} \in I_v\ | \ \psi_1,\ldots,\psi_t\right) = \psi_v\frac{S_v}{S_{h-1}} = \psi_{v}\prod_{h=v+1}^{k-1}(1-\psi_h).
	}
	\end{enumerate}
\end{Lemma}
The proof of Lemma~\ref{lem:attach:urngraph} follows from Definition~\ref{def:urngraph}, and the fact that $(S_k)_{k\in[t]}$ as in \eqref{for:intervalsPU} can be written as in \eqref{for:SkUrnP} (see the proof of \cite[Theorem 2.1]{berger2014}).

Before proving Lemma~\ref{lem:finiteedges}, we state a second result on the concentration of the positions $\{S_k\}_{k\in[t]}$ in the urn graph $(\PU_t)_{t\in\N}$. In particular, it shows that these positions concentrate around deterministic values:
\begin{Lemma}[Position concentration in $\mathrm{PU}_t$]
\label{lem:urnpos}
Consider a P\'olya urn graph as in Definition~\ref{def:urngraph}. Let $\chi = (m+\delta)/(2m+\delta)$. Then, for every $\omega,\varepsilon>0$ there exists $N_0 = N_0(\omega,\varepsilon)\in\N$ such that, for every $t\geq N_0$, 
\eqn{
\label{for:urnposition}
	\pr\bigg(\bigcap_{ i=N_0}^t\bigg\{ \left|S_i-\left(\frac{i}{t}\right)^\chi\right|\leq \omega\left(\frac{i}{t}\right)^\chi\bigg\}\bigg)\geq 1-\varepsilon
}
and, for $t$ large enough,
\eqn{
\label{for:urnpositionTot}
	\pr\left(\max_{i\in[t]}\left| S_i-\left(\frac{i}{t}\right)^\chi\right|\geq \omega\right)\leq \varepsilon
}
As a consequence, as $t\rightarrow\infty$, 
\eqn{
	\label{for:urnposition:conv}
		\max_{i\in[t]}\left| S_i-\left(\frac{i}{t}\right)^\chi\right|\stackrel{\pr}{\longrightarrow}0.
}
\end{Lemma}
The proof of Lemma \ref{lem:urnpos} is given in \cite[Lemma 3.1]{berger2014}.

\subsection{Proof of Lemma \ref{lem:finiteedges}}
\label{sec:proof:lem:finiteedges}
We now prove Lemma \ref{lem:finiteedges}, starting with the proof of \eqref{for:lem:finiteedges}. Fix $\sub{u}_\ell,\sub{v}_\ell,\sub{j}_\ell$. In the proof, we denote $M_\ell(\sub{u}_\ell,\sub{v}_\ell,\sub{j}_\ell)$ simply by $M_\ell$ to keep notation light. 
We use the fact that the P\'olya urn graph $\PU_t$ and $\PA_t$ have the same distribution and evaluate $\pr\left(M_\ell \subseteq E(\PU_t)\right)$. We consider $\ell$ distinct labeled edges, so we can use \eqref{for:edgeUrnP} to write
\eqn{
\label{for:finiteedges:0}
	\pr\left(M_\ell \subseteq E(\PU_t)\ | \ \psi_1,\ldots,\psi_t\right) = \prod_{l=1}^\ell \psi_{u_l}\frac{S_{u_l}}{S_{v_l-1}}.
}
%We point out that \eqref{for:finiteedges:0} hold since we assume that the labeled edges are distinct. 
Now fix $\varepsilon>0$. Define $\Ee_\varepsilon := \{\max_{i\in[t]}\left| S_i-\left(\frac{i}{t}\right)^\chi\right|\leq\varepsilon\}$. By \eqref{for:urnposition:conv}, and the fact that the product of the random variables in \eqref{for:finiteedges:0} is bounded by 1,
\eqn{
\label{for:finiteedges:1}
	\E\bigg[\prod_{l=1}^\ell \psi_{u_l}\frac{S_{u_l}}{S_{v_l-1}}\bigg] = \E\bigg[\I_{\Ee_\varepsilon}\prod_{l=1}^\ell \psi_{u_l}\frac{S_{u_l}}{S_{v_l-1}}\bigg] + o(1).
}
On the event $\Ee_\varepsilon$, we have, for every $l\in[\ell]$, 
\eqn{
\label{for:finiteedges:2}
	(1-\varepsilon)\left(\frac{u_l}{v_l}\right)^\chi \leq \frac{S_{u_l}}{S_{v_l-1}}\leq (1+\varepsilon)\left(\frac{u_l}{v_l}\right)^\chi,
}
where in \eqref{for:finiteedges:2} we have replaced $v_l-1$ with $v_l$ with a negligible error. Notice that since $v_l$ is always the source of the edge, this implies $v_l\geq 2$, therefore this is allowed. Using \eqref{for:finiteedges:2} in \eqref{for:finiteedges:1} we obtain  
\eqn{
\label{for:finiteedges:3}
\begin{split}
	(1-\varepsilon)^\ell \prod_{l=1}^\ell \left(\frac{u_l}{v_l}\right)^\chi\E\bigg[\I_{\Ee_\varepsilon}\prod_{l=1}^\ell \psi_{u_l}\bigg]
		&\leq \pr\left(M_\ell \subseteq E(\PU_t)\right) \\
		&\leq 
	(1+\varepsilon)^\ell \prod_{l=1}^\ell \left(\frac{u_l}{v_l}\right)^\chi\E\bigg[\I_{\Ee_\varepsilon}\prod_{l=1}^\ell \psi_{u_l}\bigg].
\end{split}
}
Even though $\psi_1\ldots,\psi_t$ depend on $\Ee_\varepsilon$,  it is easy to show that we can ignore $\I_{\Ee_\varepsilon}$ in \eqref{for:finiteedges:3} and obtain a similar bound. Therefore
\eqn{
\label{for:finiteedges:4}
	(1-\varepsilon)^\ell \prod_{l=1}^\ell \left(\frac{u_l}{v_l}\right)^\chi\E\bigg[\prod_{l=1}^\ell \psi_{u_l}\bigg]\leq 
	\pr\left(M_\ell \subseteq E(\PU_t)\right)\leq 
	(1+\varepsilon)^\ell \prod_{l=1}^\ell \left(\frac{u_l}{v_l}\right)^\chi\E\bigg[\prod_{l=1}^\ell \psi_{u_l}\bigg]
}
What remains is to evaluate the expectation in \eqref{for:finiteedges:4}. We assumed to have $\ell$ distinct edges, that does not imply that the vertices $u_1,v_1,\ldots,u_\ell,v_\ell$ are distinct. The expectation in \eqref{for:finiteedges:4} depends only on the receiving vertices of the $\ell$ edges, namely $u_1,\ldots,u_\ell$. 

Let $\bar{u}_1,\ldots, \bar{u}_k$ denote the  $k\leq \ell$ distinct elements that appear among $ u_1,\ldots,u_\ell$. For $h\in[k]$, the vertex $\bar{u}_h$ appears in the product inside the expectation in \eqref{for:finiteedges:4} with multiplicity $\din_h$, which is the degree of vertex $\bar{u}_k$ in the subgraph defined by $M_\ell$. As a consequence, we can write
\eqn{
\label{for:finiteedges:42}
	\E\bigg[\prod_{l=1}^\ell \psi_{u_l}\bigg] = \E\bigg[\prod_{h=1}^k \psi_{\bar{u}_h}^{\din_h}\bigg] = \prod_{h=1}^k\E\bigg[\psi_{\bar{u}_h}^{\din_h}\bigg],
}
where in \eqref{for:finiteedges:42} we have used the fact that $\psi_1,\ldots,\psi_t$ are all independent.  Notice that $\E[\psi_1^d] = 1$ for all $d\geq 0$, since $\psi_1\equiv 1$. Therefore, if  $\bar{u}_h=1$ for some $h\in[k]$, $\E[\psi_{\bar{u}_h}^d] = 1$ and the terms depending on the first vertex contribute to the expectation in \eqref{for:finiteedges:42} by a constant. 

For the terms where  $\bar{u}_h\geq2$, recall that, if $X(\alpha,\beta)$ is a Beta random variable, then, for any integer $d\in\N$,  
$$
	\E[X(\alpha,\beta)^d] = \frac{\alpha(\alpha+1)\cdots(\alpha+d-1)}{(\alpha+\beta)(\alpha+\beta+1)\cdots(\alpha+\beta+d-1)}.
$$ 
Since $\psi_{\bar{u}_h}$ is Beta distributed with parameters $m+\delta$ and $2(\bar{u}_h-3)+(\bar{u}_h-1)\delta$, 
\eqn{
\label{for:finiteedges:5}
\begin{split}
	\E\bigg[\psi_{\bar{u}_h}^{\din_h}\bigg] & = 
		\frac{(m+\delta)\cdots(m+\delta+\din_h-1)}
			{[m(2\bar{u}_h-2)+\bar{u}_h\delta]\cdots[m(2\bar{u}_h-2)+\bar{u}_h\delta+\din_h-1]}\\
			& = \bar{u}_h^{-\din_h}\frac{(m+\delta)\cdots(m+\delta+\din_h-1)}
			{[2m+\delta-(2m)/\bar{u}_h]\cdots[2m+\delta+(\din_h-1-2m)/\bar{u}_h]}.
\end{split}
}
Notice that if $\bar{u}_h\geq 2$, uniformly in $t$ and the precise choice of the $\ell$ edges,
$$
	(m+\delta)^{-\ell}\leq \bigg([2m+\delta-(2m)/\bar{u}]\cdots[2m+\delta+(\din_h-1-2m)/\bar{u}]\bigg)^{-1}\leq (2m+\delta+\ell)^{-\ell}.
$$
As a consequence, we can find two constants $c_1(m,\delta, \ell), c_2(m,\delta, \ell)$ such that
\eqn{
\label{for:finiteedges:6}
		c_1(m,\delta, \ell)\prod_{h=1}^k\bar{u}_h^{-\din_h}\leq \prod_{h=1}^k\E\bigg[\psi_{\bar{u}_h}^{\din_h}\bigg]\leq c_2(m,\delta, \ell)\prod_{h=1}^k\bar{u}_h^{-\din_h}.
}
We now use \eqref{for:finiteedges:6} in \eqref{for:finiteedges:4} to obtain 
\eqn{
\label{for:finiteedges:7}
\begin{split}
	 c_1(m,\delta, \ell)(1-\varepsilon)^\ell  \prod_{l=1}^\ell \left(\frac{u_l}{v_l}\right)^\chi \prod_{h=1}^k\bar{u}_h^{-\din_h} & \leq \pr\left(M_\ell\subseteq E(\PU_t)\right)\\
	& \leq  c_2(m,\delta, \ell)(1+\varepsilon)^\ell  \prod_{l=1}^\ell \left(\frac{u_l}{v_l}\right)^\chi \prod_{h=1}^k\bar{u}_h^{-\din_h}.
	\end{split}
}
In \eqref{for:finiteedges:7} we can just rename the constants   $c_1(m,\delta, \ell) = c_1(m,\delta, \ell)(1-\varepsilon)^\ell$ and $c_2(m,\delta, \ell) = c_2(m,\delta, \ell)(1+\varepsilon)^\ell$.
Since $\din_h$ is the multiplicity of vertex $\bar{u}_h$ as receiving vertex, we can write
$$
	\prod_{h=1}^k\bar{u}_h^{-\din_h} = \prod_{l=1}^\ell u_l^{-1}.
$$
Combining this with \eqref{for:finiteedges:7} completes the proof of \eqref{for:lem:finiteedges}.

The proof of~\eqref{for:lem:finiteedges:3} follows immediately from \eqref{for:lem:finiteedges} and the definition of the set $J(\sub{u}_\ell, \sub{v}_\ell)$ in \eqref{for:lem:finiteedges:2}. In fact, we can write
$$
	\E[|J(\sub{u}_\ell, \sub{v}_\ell)|] = \sum_{\sub{j}_\ell\in[m]^\ell}\pr\left(M_\ell(\sub{u}_\ell,\sub{v}_\ell,\sub{j}_\ell)\subseteq E(\PA_t)\right).
$$
Recall that $\pr\left(M_\ell(\sub{u}_\ell,\sub{v}_\ell,\sub{j}_\ell)\subseteq E(\PA_t)\right)$ is {\em independent of the labels} $\sub{j}_\ell$.
For a fixed set of source and target vertices $\sub{u}_\ell$ and $\sub{v}_\ell$, there is only a finite combination of labels $\sub{j}_\ell$ such that the subgraph defined by $M_\ell(\sub{u}_\ell,\sub{v}_\ell,\sub{j}_\ell)$ is attainable in the sense of Definition \ref{def:attainable}. 
In fact, the number of such labels $\sub{j}_\ell$ is larger than one (since the corresponding subgraph is attainable), and less than $m^\ell$ (the total number of elements of $[m]^\ell$). 
As a consequence, taking $\hat{c}_1 = c_1$ and $\hat{c}_2 = c_2m^\ell$ proves~\eqref{for:lem:finiteedges:3}.
\qed

\section{Proof of Theorem~\ref{th:subgraph:expectation}}\label{sec:proofs}
To prove Theorem~\ref{th:subgraph:expectation}, we write the expected number of subgraphs as multiple integrals. W.l.o.g.\ we assume throughout this section that $\pi$ is the identity permutation, so that the vertices of $H$ are labeled as $1,\dots,k$, and therefore drop the dependence of the quantities on $\pi$. We first prove a lemma that states that two integrals that will be important in proving Theorem~\ref{th:subgraph:expectation} are finite:
\begin{lemma}\label{lem:finiteint}
	Let $H$ be a subgraph such that the optimum of~\eqref{eq:opteq} is attained by $s_1,\dots,s_r$. Then,
	\begin{align}
	A_1(H)&:=\int_{1}^{\infty}u_1^{\beta(1)}\int_{u_1}^{\infty}u_2^{\beta(2)}\cdots \int_{u_{s-1}}^{\infty}u_{s_1}^{\beta(s_1)}\dd u_{s_1}\cdots \dd u_1<\infty,\\
	A_2(H)&:=\int_{0}^{1}u_{k}^{\beta({k})}\int_0^{u_{k}}u_{k-1}^{\beta({k-1})}\cdots \int_0^{u_{s_r+1}}u_{s_r+1}^{\beta(s_r+1)}\dd u_{s_r+1}\cdots \dd u_k<\infty.
	\end{align}
\end{lemma}
\begin{proof}
	The first integral is finite as long as
	\begin{equation}\label{eq:intfinite1}
	z+\sum_{i=s_1-z}^{s_1}\beta(i)<0
	\end{equation}
	for all $z\in[s_1]$. Suppose that~\eqref{eq:intfinite1} does not hold for some $z^*\in[s_1]$. Then, the difference between the contribution to~\eqref{eq:opteq} for $\tilde{s}=s_1-z^*$ and $s_1$ is
	\begin{equation}
	-(s_1-z^*)+\sum_{i=s_1-z^*}^{k}\beta(i)+s_1-\sum_{i=s_1}^k \beta(i)= z^*+\sum_{i=s_1-z^*}^{s_1}\beta(i)\geq 0,
	\end{equation}
	which would imply that $s_1-z^*$ is also an optimizer of~\eqref{eq:opteq}, which is in contradiction with $s_1$ being the smallest optimum. Thus,~\eqref{eq:intfinite1} holds for all $r\in[s]$ and $A_1(H)<\infty$. 
	
	The second integral is finite as long as
	\begin{equation}
	z-s_r+\sum_{i=s_r+1}^{z}\beta(i)>0
	\end{equation}
	for all $z\in\{s_r+1,\dots,k\}$. Suppose that this does not hold for some $z^*\in\{s_r+1,\dots,k\}$. Set $\tilde{s}=z^*>s_r$. Then, the difference between the contribution to~\eqref{eq:opteq} for $\tilde{s}=z^*$ and $s_r$ is
	\begin{equation}
	-z^*+s_r-\sum_{i=s_r+1}^{z^*}\beta(i)\geq 0,
	\end{equation}
	which is a contradiction with $s_r$ being the largest optimizer. Therefore, $A_2(H)<\infty$.
\end{proof}

We now use this lemma to prove Theorem~\ref{th:subgraph:expectation}:
\begin{proof}[Proof of Theorem~\ref{th:subgraph:expectation}]
	Again, we assume that $\pi$ is the identity mapping, so that we may drop all dependencies on $\pi$. 
	Suppose the optimal solution to~\eqref{eq:opteq} is attained by $s_1,s_2,\dots,s_r$ for some $r\geq1$. Let the $\ell$ edges of $H$ be denoted by $(u_l,v_l)$ for $l\in[\ell]$. Let $N_t(H,i_1,\ldots,i_k)$ denote the number of times subgraph $H$ is present on vertices $i_1,\dots,i_k$. 
	We then use Lemma~\ref{lem:finiteedges}, which proves that, for some $0<C<\infty$, 
	\begin{equation}\label{eq:expnh}
	\begin{aligned}
		\Exp{N_t(H)} &= \sum_{i_1<\cdots<i_k\in[t]}\Exp{N_t(H,i_1,\dots,i_k)}\\
		& \leq C\sum_{i_1<\cdots <i_k\in[t]}\prod_{l=1}^\ell i_{u_l}^{\chi-1}i_{v_l}^{-\chi} = C\sum_{i_1<\cdots<i_k\in[t]}\prod_{q=1}^ki_q^{\beta(q)}.
	\end{aligned}
	\end{equation}
	
	We then bound the sums by integrals as
	\begin{equation}\label{eq:ENHub}
	\begin{aligned}[b]
	\Exp{N_t(H)}& \leq \tilde{C} \int_{1}^{t}u_1^{\beta(1)}\dots \int_{u_{k-1}}^{t}u_k^{\beta(k)}\dd u_k\cdots \dd u_1\\
	& \leq \tilde{C}  \int_{1}^{\infty}u_1^{\beta(1)}\dots \int_{u_s-1}^{\infty}u_{s_1}^{\beta(s_1)}\dd u_{s_1}\dots \dd u_{1}\\
	& \quad \times \int_{1}^{t}u_{s_1+1}^{\beta(s_1+1)}\int_{u_{s_1+1}}^{\infty}u_{s_1+2}^{\beta(s_1+2)}\cdots \int_{u_{s_2-1}}^{t}u_{s_2}^{\beta(s_2)}\dd u_{s_2}\dots \dd u_{s_1+1} \\
	& \quad \times \int_{1}^{t}u_{s_2+1}^{\beta(s_2+1)}\int_{u_{s_2+1}}^{\infty}u_{s_2+2}^{\beta(s_2+2)}\cdots \int_{u_{s_3-1}}^{t}u_{s_3}^{\beta(s_3)}\dd u_{s_3}\dots \dd u_{s_2+1}\times  \cdots\\
	& \quad \times \int_{1}^{t}u_{s_{r-1}+1}^{\beta(s_{r-1}+1)}\int_{u_{s_{r-1}+1}}^{\infty}u_{s_{r-1}+2}^{\beta(s_{r-1}+2)}\cdots \int_{u_{s_r-1}}^{t}u_{s_r}^{\beta(s_r)}\dd u_{s_r}\dots \dd u_{s_{r-1}+1} \\
	& \quad \times \int_{0}^{t}u_{s_r+1}^{\beta(s_r+1)}\int_{s_r+1}^{t}u_{s_r+2}^{\beta(s_r+2)}\cdots \int_{u_{k-1}}^{t}u_k^{\beta(k)}\dd u_k\cdots \dd u_{s_r+1},
	\end{aligned}
	\end{equation}
	for some $0<\tilde{C}<\infty$. 
	The first set of integrals is finite by Lemma~\ref{lem:finiteint} and independent of $t$. For the last set of integrals, we obtain
	\begin{equation}\label{eq:ENHub1}
	\begin{aligned}[b]
	&\int_{0}^{t}u_{s_r+1}^{\beta({s_r+1})}\int_{u_{s_r+1}}^{t}u_{s_r+2}^{\beta({s+2})}\cdots \int_{u_{k-1}}^{t}u_k^{\beta(k)}\dd u_k\cdots \dd u_{s_r+1}\\
	&=	t^{k-s_r+\sum_{i=s_r+1}^{k}\beta(i)}\int_{0}^{1}w_{s_r+1}^{\beta({s_r+1})}\int_{w_{s_r+1}}^{1}w_{s_r+2}^{\beta({s_r+2})}\cdots \int_{w_{k-1}}^{1}w_k^{\beta(k)}\dd w_k\cdots \dd w_{s_r+1}\\
	& = Kt^{k+B(H)},
	\end{aligned}
	\end{equation}
	for some $0<K<\infty$, where we have used the change of variables $w=u/t$ and Lemma~\ref{lem:finiteint}. For $r=1$, this finishes the proof, because then the middle integrals in~\eqref{eq:ENHub} are empty. We now investigate the behavior of the middle sets of integrals for $r>1$. 
	Because the optimum to~\eqref{eq:opteq} is attained for $s_1$ as well as $s_2$,
	\begin{equation}\label{eq:s1s2}
	-s_1+\sum_{i=s_1+1}^{k}\beta(i)+s_2-\sum_{i=s_2+1}^{k}\beta(i)=s_2-s_1+\sum_{i=s_1+1}^{s_2}\beta(i)=0.
	\end{equation}
	Therefore, when $s_2=s_1+1$, the second set of integrals in~\eqref{eq:ENHub} equals
	\begin{equation}
	\int_{1}^{t}u_{s_1}^{-1}\dd u_{s_1}=\log (t).
	\end{equation}
	Now suppose that $s_1<s_2+1$. Then, any $\tilde{s}\in[s_1+1,s_2-1]$ is a non-optimal solution to~\eqref{eq:opteq}, and therefore
	\begin{equation}
	-s_2+\sum_{i=s_2+1}^{k}\beta(i)+\tilde{s}-\sum_{i=\tilde{s}+1}^{k}\beta(i)=\tilde{s}-s_2-\sum_{i=\tilde{s}+1}^{s_2}\beta(i)>0,
	\end{equation}
	or
	\begin{equation}\label{eq:sbetween}
	\sum_{i=\tilde{s}+1}^{s_2}\beta(i) <s_2-\tilde{s}.
	\end{equation}
	This implies that
	\begin{equation}\label{eq:ENHubmid}
	\begin{aligned}[b]
	&\int_{1}^{t}u_{s_1+1}^{\beta(s_1+1)}\int_{u_{s_1+1}}^{\infty}u_{s_1+2}^{\beta(s_1+2)}\cdots \int_{u_{s_2-1}}^{t}u_{s_2}^{\beta(s_2)}\dd u_{s_2}\ldots \dd u_{s_1+1}\\
	&\quad  = \tilde{K}\int_{1}^{t}u_{s_1+1}^{\sum_{i=s_1+1}^{s_2}\beta(i)+s_2-s_1-1}\dd u_{s_1+1}\\
	& \quad= \tilde{K}\int_{1}^{t}u_{s_1+1}^{-1}\dd u_{s_1+1}=\tilde{K}\log (t),
	\end{aligned}
	\end{equation}
	for some $0<C<\infty$. A similar reasoning holds for the other integrals, so that combining~\eqref{eq:ENHub},~\eqref{eq:ENHub1} and~\eqref{eq:ENHubmid} yields
	\begin{equation}
	\lim_{t\to\infty}\frac{\Exp{N_t(H)}}{t^{k+B(H)}\log^{r-1}(t)}\leq C_2,
	\end{equation}
	for some $0<C_2<\infty$.
	
	We now proceed to prove a lower bound on the expected number of subgraphs. Again, by Lemma~\ref{lem:finiteedges} and lower bounding the sums by integrals as in~\eqref{eq:expnh}, we obtain that, for some $0<C<\infty$
	\begin{equation}
		\Exp{N_t(H)}\geq C \int_{1}^{t}u_1^{\beta(1)}\dots \int_{u_{k-1}}^{t}u_k^{\beta(k)}\dd u_k\cdots \dd u_1.
	\end{equation}
	 Fix $\varepsilon>0$. We investigate the contribution where vertices $1,\dots,s_1$ have index in $[1,1/\varepsilon]$, vertices $s_1+1,\dots,s_2$ have index in $[1/\varepsilon,\varepsilon t^{1/r}]$, vertices $s_2+1,\dots,s_3$ have index in $[t^{1/r},\varepsilon t^{2/r}]$ and so on, and vertices $s_r+1,\dots, s_k$ have index in $[\varepsilon t,t]$. Thus, we bound
	\eqan{\label{eq:LBnonunique}
	%\begin{aligned}[b]
	\Exp{N_t(H)}&  \geq C\int_{1}^{1/\varepsilon}u_1^{\beta(1)}\int_{u_1}^{1/\varepsilon}u_2^{\beta(2)}\cdots\int_{u_{s_1-1}}^{1/\varepsilon}u_{s_1}^{\beta(s)}\dd u_{s_1}\dots\dd u_1\nonumber \\
	& \quad \times \int_{1/\varepsilon}^{\varepsilon t^{1/r}}u_{s_1+1}^{\beta(s_1+1)}\int_{u_{s_1+1}}^{u_{s_1+1}/\varepsilon}u_{s_1+2}^{\beta(s_1+2)}\cdots \int_{u_{s_2-1}}^{u_{s_2-1}/\varepsilon}u_{s_2}^{\beta(s_2)}\dd u_{s_2}\dots\dd u_{s_1+1}\nonumber \\
	& \quad \times \int_{t^{1/r}}^{\varepsilon t^{2/r}}u_{s_2+1}^{\beta(s_2+1)}\int_{u_{s_2+1}}^{u_{s_2+1}/\varepsilon}u_{s_2+2}^{\beta(s_2+2)}\cdots \int_{u_{s_3-1}}^{u_{s_3-1}/\varepsilon}u_{s_3}^{\beta(s_3)}\dd u_{s_3}\dots\dd u_{s_2+1}\times\cdots \nonumber \\
	& \quad \times \int_{t^{(r-2)/r}}^{\varepsilon t^{(r-1)/r}}u_{s_{r-1}+1}^{\beta(s_{r-1}+1)}\int_{u_{s_{r-1}+1}}^{u_{s_{r-1}+1}/\varepsilon}u_{s_{r-1}+2}^{\beta(s_{r-1}+2)}\cdots \int_{u_{s_r-1}}^{u_{s_r-1}/\varepsilon}u_{s_r}^{\beta(s_r)}\dd u_{s_r}\dots\dd u_{s_{r-1}+1}\nonumber \\
	&\quad \times \int_{\varepsilon t}^{t}u_{s_r+1}^{\beta(s_r+1)}\int_{u_{s_r+1}}^{t}u_{s_r+2}^{\beta(s_r+2)}\cdots \int_{u_{k-1}}^{t}u_k^{\beta(k)}\dd u_k\dots\dd u_{s_r+1}
	%\end{aligned}
	}
	The first set of integrals equals $A_1(H)$ plus terms that vanishes as $\varepsilon$ becomes small by Lemma~\ref{lem:finiteint}. For the last set of integrals, we use the change of variables $w=u/t$ to obtain
	\begin{equation}\label{eq:LB1}
	\begin{aligned}[b]
	& \int_{\varepsilon t}^{t}u_{s_r+1}^{\beta({s_r+1})}\int_{u_{s_r+1}}^{t}u_{s_r+2}^{\beta({s_r+2})}\cdots \int_{u_{k-1}}^{t}u_k^{\beta(k)}\dd u_k\ldots \dd u_{s_r+1}\\
	& =	t^{k-s_r+\sum_{i=s_r+1}^{k}\beta(i)}\int_{\varepsilon }^{1}w_{s_r+1}^{\beta({s_r+1})}\int_{w_{s_r+1}}^{1}w_{s_r+2}^{\beta({s_r+2})}\cdots \int_{w_{k-1}}^{1}w_k^{\beta(k)}\dd w_k\ldots \dd w_{s_r+1}\\
	& = t^{k+B(H)}(A_2(H)-h_1(\varepsilon)),
	\end{aligned}
	\end{equation}
	for some function $h_1(\varepsilon)$. By Lemma~\ref{lem:finiteint} $h_1(\varepsilon)$ satisfies $\lim_{\varepsilon\to 0}h_1(\varepsilon)= 0$. 
	Again, if $r=1$, the middle sets of integrals in~\eqref{eq:LBnonunique} are empty, so we are done. 
	
	We now investigate the second set of integrals in~\eqref{eq:LBnonunique} for $r>1$. Using the substitution $w_{s_1+1}=u_{s_1+1}$ and $w_i=u_i/u_{i-1}$ for $i>s_1+1$, we obtain
	\eqan{\label{eq:LBmiddle}
	%\begin{aligned}[b]
	&\int_{1/\varepsilon}^{\varepsilon t^{1/r}}u_{s_1+1}^{\beta(s_1+1)}\int_{u_{s_1+1}}^{u_{s_1+1}/\varepsilon}u_{s_1+2}^{\beta(s_1+2)}\cdots \int_{u_{s_2-1}}^{u_{s_2-1}/\varepsilon}u_{s_2}^{\beta(s_2)}\dd u_{s_2}\dots\dd u_{s_1+1}\\
	& = \int_{1/\varepsilon}^{\varepsilon t^{1/r}}w_{s_1+1}^{s_2-s_1-1+\sum_{i=s_1+1}^{s_2}\beta(i)}\dd w_{s_1+1}\int_{1}^{1/\varepsilon}w_{s_1+2}^{s_2-s_1-2+\sum_{i=s_1+2}^{s_2}\beta(i)}\dd w_{s_2+1}\cdots \int_{1}^{1/\varepsilon}w_{s_2}^{\beta(s_2)}\dd w_{s_2}.
	\nonumber
	%\end{aligned}
	}
	The first integral equals by~\eqref{eq:s1s2}
	\begin{equation}
	\int_{1/\varepsilon}^{\varepsilon t^{1/r}}w_{s_1+1}^{-1}\dd w_{s_1+1} = \frac{1}{r}\log(t)+\log(\varepsilon^2).
	\end{equation} 
	The integrand in all other integrals in~\eqref{eq:LBmiddle} equals $w_i^{\gamma_i}$ for some $\gamma_i<-1$ by~\eqref{eq:sbetween}. Therefore, these integrals equal a constant plus a function of $\varepsilon$ that vanishes as $\varepsilon$ becomes small so that
	\begin{equation}\label{eq:LBmiddleend}
	\begin{aligned}[b]
	&\int_{1/\varepsilon}^{\varepsilon t^{1/r}}u_{s_1+1}^{\beta(s_1+1)}\int_{u_{s_1+1}}^{u_{s_1+1}/\varepsilon}u_{s_1+2}^{\beta(s_1+2)}\cdots \int_{u_{s_2-1}}^{u_{s_2-1}/\varepsilon}u_{s_2}^{\beta(s_2)}\dd u_{s_2}\dots\dd u_{s_1+1}\\
	& = \left(\frac{1}{r}\log(t)+\log(\varepsilon^2)\right)\left(K+h_2(\varepsilon)\right),
	\end{aligned}
	\end{equation}
	for some $0<K<\infty$ and some $h_2(\varepsilon)$ such that $\lim_{\varepsilon\to 0}h_2(\varepsilon)=0$. The other integrals in~\eqref{eq:LBnonunique} can be estimated similarly. 
	
	Combining~\eqref{eq:LBnonunique},~\eqref{eq:LB1} and~\eqref{eq:LBmiddleend} we obtain
	\begin{equation}
	\lim_{t\to\infty}\frac{\Exp{N_t(H)}}{t^{k+B(H)}\log^{r-1}(t)}\geq C_1+h(\varepsilon),
	\end{equation}
	for some constant $0<C_1<\infty$ and some function $h(\varepsilon)$ such that $\lim_{\varepsilon\to 0}h(\varepsilon)=0$. Taking the limit for $\varepsilon\to 0$ then proves the theorem.
\end{proof}

\section{Proof of Theorem \ref{th:numberT}}
\label{sec:numberT:proof}
Fix $m\geq 2$ and $\delta>-m$. The first step of the proof consists of showing that
\eqn{
\label{for:Texpected}
\begin{split}
	\E[\triangle_t]  = 
	& \frac{\tau-2}{\tau-1}\frac{m^2(m-1)(m+\delta)(m+\delta+1)}{(2m+\delta)^2} \\
	& \times \sum_{u=1}^{t-2}[(u-(2m)/(2m+\delta))(u-(2m-1)/(2m+\delta))]^{-1}
			\\
			&\times\frac{\Gamma(u+2-(2m)/(2m+\delta)}{\Gamma(u+2-(3m+\delta)/(2m+\delta))}
				\frac{\Gamma(u+2-(2m-1)/(2m+\delta))}{\Gamma(u+2-(3m+\delta-1)/(2m+\delta))}\\
	& \times \sum_{v=u+1}^{t-1}(v-(3m+\delta-1)/(2m+\delta))^{-1}\\
	& \times \sum_{w=v+1}^{t}\frac{\Gamma(w-(3m+\delta)/(2m+\delta))}{\Gamma(w-(2m)/(2m+\delta))}
						\frac{\Gamma(w-(3m+\delta-1)/(2m+\delta))}{\Gamma(w-(2m-1)/(2m+\delta))}.
\end{split}
}
 We can write
\eqn{
\label{for:labeltriang}
	\triangle_t := \sum_{u=1}^{t-2}\sum_{v=u+1}^{t-1}\sum_{w=v+1}^t \sum_{j_1\in[m]}\sum_{j_2,j_3\in[m]}\I\{u\stackrel{j_1}{\leftarrow} v,\ u\stackrel{j_2}{\leftarrow}w, \ v\stackrel{j_3}{\leftarrow} w\}.
}
Since there are $m^2(m-1)$ possible choices for the edges $j_1,j_2,j_3$, 
\eqn{
\label{for:Texpected:1}
	\E[\triangle_t] = m^2(m-1)\sum_{u=1}^{t-2}\sum_{v=u+1}^{t-1}\sum_{w=v+1}^t
		\E\left[\psi_u\frac{S_u}{S_{v-1}}\psi_u\frac{S_u}{S_{w-1}}\psi_v\frac{S_v}{S_{w-1}}\right].
}
Recalling \eqref{for:edgeUrnP}, we can write every term in the sum in \eqref{for:Texpected:1} as
\eqn{
\label{for:Texpected:rem}
	\E\bigg[\bigg(\psi_u\prod_{h=u+1}^{v-1}(1-\psi_h)\bigg)\bigg(\psi_u\prod_{k=u+1}^{w-1}(1-\psi_k)\bigg)\bigg(\psi_v\prod_{l=v+1}^{w-1}(1-\psi_l)\bigg)\bigg].	
}
Since the random variables $\psi_1,\ldots, \psi_t$ are independent, we can factorize the expectation to obtain
\eqn{
\label{for:Texpected:2}
	\E[\psi_u^2]\E[\psi_v(1-\psi_v)] \prod_{k=u+1,k\neq v}^{w-1}\E[(1-\psi_k)^2] =\E[\psi_u^2]\frac{\E[\psi_v(1-\psi_v)]}{\E[(1-\psi_v)^2]}  \prod_{k=u+1}^{w-1}\E[(1-\psi_k)^2] .
}
Recall that, for a Beta random variable $X(\alpha,\beta)$, we have
\eqn{
\label{for:Texpected:2:extra}
	\begin{split}
		\E[X] & = \frac{\alpha}{\alpha+\beta},\\
		\E[X(1-X)] & = \frac{\alpha\beta}{(\alpha+\beta)(\alpha+\beta+1)},\\
		\E[X^2] & = \frac{\alpha(\alpha+1)}{(\alpha+\beta)(\alpha+\beta+1)},
	\end{split}
}
and $1-X(\alpha,\beta)$ is distributed as $X(\beta,\alpha)$. Using \eqref{for:Texpected:2:extra}, we can rewrite \eqref{for:Texpected:2} in terms of the parameters of $\psi_1,\ldots, \psi_t$. 
Since $\psi_k$ has parameters $\alpha = m+\delta$ and $\beta= \beta_k = m(2k-3)+(k-1)\delta$, the first term in \eqref{for:Texpected:2} can be written as
\eqn{
\label{for:Texpected:3}
\begin{split}
		\E[\psi_u^2]  & \frac{(m+\delta)(m+\delta+1)}{(m(2u-2)+u\delta)(m(2u-2)+u\delta+1)}
		\\ & = \frac{(m+\delta)(m+\delta+1)}{(2m+\delta)^2}\bigg[(u-2m/2m+\delta)(u-(2m-1)/(2m+\delta))\bigg]^{-1}.
\end{split}
}
For the second term, we have 
\eqn{
\label{for:Texpected:4}
	\frac{\E[\psi_v(1-\psi_v)]}{\E[(1-\psi_v)^2]}  =\frac{m+\delta}{m(2v-3)+(v-1)\delta} =  \frac{\tau-2}{\tau-1} (v-(3m+\delta-1)/(2m+\delta))^{-1}.
}
The last product in \eqref{for:Texpected:2}, for $k=u+1,\ldots,w-1$ results in
\eqn{
\label{for:Texpected:4:1}
\begin{split}
	\E[(1-\psi_k)^2] & =\frac{(m(2k-3)+(k-1)\delta) (m(2k-3)+(k-1)\delta+1)}{(m(2k-2)+k\delta) (m(2k-2)+k\delta+1)}\\
	& =   \frac{k-(3m+\delta)/(2m+\delta)}{k-2m/(2m+\delta)}\frac{k-(3m+\delta-1)/(2m+\delta)}{k-(2m-1)/(2m+\delta)}.
\end{split}
}
Using the recursive property $\Gamma(a+1) = a\Gamma(a)$ of the Gamma function, 
\eqan{
\label{for:Texpected:5}
%\begin{split}
	\prod_{k=u+1}^{w-1}\E[(1-\psi_k)^2] = & 
	\frac{\Gamma(u+2-(2m)/(2m+\delta)}{\Gamma(u+2-(3m+\delta)/(2m+\delta))}
				\frac{\Gamma(u+2-(2m-1)/(2m+\delta))}{\Gamma(u+2-(3m+\delta-1)/(2m+\delta))}\nonumber \\
	& \times \frac{\Gamma(w-(3m+\delta)/(2m+\delta))}{\Gamma(w-(2m)/(2m+\delta))}
						\frac{\Gamma(w-(3m+\delta-1)/(2m+\delta))}{\Gamma(w-(2m-1)/(2m+\delta))}.
%\end{split}
}
Equation \eqref{for:Texpected:1} follows by combining~\eqref{for:Texpected:2}, \eqref{for:Texpected:3}, \eqref{for:Texpected:4},  \eqref{for:Texpected:4:1} and \eqref{for:Texpected:5}.

The last step of the proof is to evaluate the sum in \eqref{for:Texpected:1}, and combining the result with the multiplicative constant in front in \eqref{for:Texpected:1}. By Stirling's formula
$$
	\frac{\Gamma(x+a)}{\Gamma(x+b)} = x^{a-b}(1+O(1/x)).
$$
As a consequence, recalling that $\chi = (m+\delta)/(2m+\delta)$, the sum in \eqref{for:Texpected:1} can be written as
\eqn{
\label{for:Texpected:6}
	\sum_{u=1}^{t-2}u^{2\chi-2}(1+O(1/u))\sum_{v=u+1}^{t-1}v^{-1}(1+O(1/v))\sum_{w=v+1}^t w^{-2\chi}(1+O(1/w)).
}
We can approximate the sum in \eqref{for:Texpected:6} with the corresponding integral using Euler-Maclaurin formula, thus obtaining
\eqn{
\label{for:Texpected:7}
	\int_1^t u^{2\chi-2}du\int_u^t v^{-1}dv\int_v^t w^{-2\chi}dw.
}
As $t\rightarrow\infty$, the order of magnitude of the integral in \eqref{for:Texpected:7} is predicted by Theorem~\ref{th:subgraph:expectation}. If we evaluate the integral, then we obtain that the coefficient of the dominant term in \eqref{for:Texpected:7} is $(2m+\delta)^2/\delta^2$ for $\tau>2, \tau\neq 3$, and 
$1/6$ for $\tau=3$. 

Putting together these coefficients with the constant in front of the sum in \eqref{for:Texpected} completes the proof of  Theorem \ref{th:numberT}. \qed

\begin{Remark}[Constant for general subgraphs]
\label{rem:subgraphs:constant}
{\em 
In the proof of Theorem \ref{th:numberT}, the hardest step is to prove  \eqref{for:labeltriang}, i.e., to find the expectation of the indicator functions in \eqref{for:Texpected}.
This is the reason why for a general ordered subgraph $(H,\pi)$ on $k$ vertices it is hard to find the explicit constant as in \eqref{for:expect:subgraph:constant}. In fact, as we have done to move from \eqref{for:Texpected:1} to \eqref{for:Texpected:rem}, it is necessary to identify precisely, for every $v\in[t]$, how many times the terms $\psi_v$ and $(1-\psi_v)$ appear in the product inside the expectations in \eqref{for:Texpected:1}. This makes the evaluation of such terms complicated.

Typically, as it shown in \eqref{for:Texpected:2}, \eqref{for:Texpected:3}, \eqref{for:Texpected:4},  \eqref{for:Texpected:4:1} and \eqref{for:Texpected:5}, the product of the constants obtained by evaluating the probability of an ordered subgraph $(H,\pi)$ being present can be written as ratios of Gamma functions. The same constants can be found using the martingale approach as in  \cite{Bol01,Szym05} and \cite[Section 8.3]{hofstad2009}, even though in this case constants are obtained through a recursive use of conditional expectation. 

We remark that our method and the martingale method are equivalent. We focused on the P\'olya urn interpretation of the graph since it highlights the dependence of the presence of edges on the {\em age} of vertices, that is directly related to the order of magnitude of degrees. 
}
\end{Remark}

\section{Conditional concentration: proof of Proposition~\ref{prop:criterion:conditional:subgraph}}
\label{sec:fluct}
In the previous sections, we have considered the order of magnitude of the expectation of the number of occurrences of ordered subgraphs in PAM. In other words, for an ordered subgraph $(H,\psi)$ we are able to identify the order of magnitude $f(t)$ of the expected number of occurrences $N_t(H,\pi)$, so that $\E[N_t(H,\pi)] = \bigO{f(t)}$.
%but we do not know the precise constant $C$ such that $\E[N_t(H,\pi)] = Cf(t)(1+o(1))$. In Section \ref{sec:numberT:proof}, using the triangle as example of ordered subgraph, we discuss why this constant $C$ is difficult to identify in general.
We now show how these orders of magnitude of the expected number of subgraphs determines the conditional convergence given in~\eqref{for:subgraph:convergence1}.

\subsection{Bound with overlapping subgraphs}
The P\'olya urn graph in Definition~\ref{def:urngraph} consists of a function of uniform random variables $(U_{v,j})_{v\in[t]}^{j\in[m]}$ and an independent sequence of Beta random variables $(\psi_v)_{v\in[t]}$. 
We can interpret the sequence $(\psi_v)_{v\in[t]}$ as a sequence of {\em intensities} associated to the vertices, where a higher intensity corresponds to a higher probability of receiving a connection. The sequence $(U_{v,j})_{v\in[t]}^{j\in[m]}$ determines the attachment of edges. In particular, conditionally on the sequence $(\psi_v)_{v\in[t]}$, every edge is present {\em independently} (but with different probabilities). 

For $t\in\N$, denote $\pr_{\sub{\psi}_t}(\cdot) = \pr(\ \cdot\  | \psi_1,\ldots,\psi_t)$, and similarly $\E_{\sub{\psi}_t}[\cdot] = \E[\ \cdot \ | \psi_1,\ldots,\psi_t]$. Furthermore, let $N_{t,\psi}(H,\pi)$ denote the number of times subgraph $(H,\pi)$ appears conditionally on the $\psi$-variables. We now apply a conditional second moment method to $N_{t,\psi}(H,\pi)$. 
We use the notation introduced in Section~\ref{sec:psubgraph}, so that every possible realization of $H$ in PAM corresponds to a finite set of edges 
$M_\ell(\sub{u}_\ell,\sub{v}_\ell,\sub{j}_\ell)$, where $\ell$ is the number of edges in $H$ such that $v_h\stackrel{j_h}{\rightarrow}u_h$, i.e., $u_h$ is the receiving vertex, and $j_h$ is the label of the edge. For simplicity, we denote the set $M_\ell(\sub{u}_\ell,\sub{v}_\ell,\sub{j}_\ell)$ by $M$. For ease of notation, we assume that $\pi$ is the identity map and drop the dependence on $\pi$. We prove the following results:

\begin{Lemma}[Bound on conditional variance]
\label{lem:bounvariance:subgraph}
	Consider subgraph $H$. Then, $\pr$-a.s., 
	\eqn{\nonumber
		\mathrm{Var}_{\sub{\psi}_t}(N_t(H))  \leq \E_{\sub{\psi}_t}[N_t(H)]+ \sum_{\hat{H}\in\hat{\mathcal{H}}}\E_{\sub{\psi}_t}[N_t(\hat{H})],
	}
	where $\hat{\mathcal{H}}$ denotes the set of all possible attainable subgraphs $\hat{H}$ that are obtained by merging two copies of $H$ such that they share at least one edge. 
%	\eqan{
%	\label{for:overlappingM}
%		\mathcal{M}:= \Big\{ (M,M') \colon & \exists \ (u,v,j) \colon (u,v,j)\in M, \ (u,v,j)\in M', \nonumber \\
%								& \ M\neq M', \ \mbox{$(M\cup M')$ defines an attainable subgraph}\Big\}.
%	}
\end{Lemma}
Lemma~\ref{lem:bounvariance:subgraph} gives a bound on the conditional variance in terms of the conditional probabilities of observing two overlapping of the subgraph $H$ at the same time. 
%The set $\mathcal{M}$ defined in \eqref{for:overlappingM} corresponds to the set of pairs of {\em distinct} realizations $M$ and $M'$ that {\em share at least one labeled edge}. 
Notice that we require these copies to overlap at at least one edge, which is different than requiring that they are disjoint (the can share one or more vertices but no edges). 

\begin{proof}[Proof of Lemma~\ref{lem:bounvariance:subgraph}]
We prove the bound in Lemma~\ref{lem:bounvariance:subgraph} by evaluating the conditional second moment of $N_t(H)$ as
\eqan{
	\E_{\sub{\psi}_t}[N_t(H)^2]  & = \E_{\sub{\psi}_t}\Big[\sum_{M,M'}\I_{\{M\subseteq E(\PA_t)\}}\I_{\{M'\subseteq E(\PA_t)\}}\Big] \nonumber \\
	 & = \sum_{M,M'}\pr_{\sub{\psi}_t}\Big(M\subseteq E(\PA_t),\ M'\subseteq E(\PA_t)\Big)\nonumber ,
}
where $M$ and $M'$ are two sets of edges corresponding to two possible realizations of the subgraph $H$. Notice that $M$ and $M'$ are not necessarily distinct. We then have to evaluate the conditional probability of having both the sets $M$ and $M'$ simultaneously present in the graph. 
As a consequence, we conditional variance in Lemma~\ref{lem:bounvariance:subgraph} can be written as
\eqn{
	\label{for:subg:variance:2}
	\sum_{M\neq M'}\pr_{\sub{\psi}_t}\Big(M\subseteq E(\PA_t),\ M'\subseteq E(\PA_t)\Big) -\pr_{\sub{\psi}_t}(M\subseteq E(\PA_t))\pr_{\sub{\psi}_t}(M'\subseteq E(\PA_t)).
}

We define
\eqan{
	\label{for:overlappingM}
	\mathcal{M}:= \Big\{ (M,M') \colon & \exists \ (u,v,j) \colon (u,v,j)\in M, \ (u,v,j)\in M', \nonumber \\
	& \ M\neq M', \ \mbox{$(M\cup M')$ defines an attainable subgraph}\Big\}.
}
We then consider two different cases, i.e., whether $(M,M')$ is in $\mathcal{M}$ or not. 
If $(M,M')\not \in \mathcal{M}$, then one of the three following situations occurs:
\begin{itemize}
	\item[{$\vartriangleright$}] $M\cup M'$ defines a subgraph that is not attainable (for instance, $M$ and $M'$ require that the same edge is attached to different vertices);
	\item[{$\vartriangleright$}] $M\cup M'$ defines a subgraph that is attainable, $M$ and $M'$ are disjoint sets of labeled edges (they are allowed to share vertices);
	\item[{$\vartriangleright$}] $M$ and $M'$ define the same attainable subgraph (so $M=M'$, thus labels of edges coincide). 
\end{itemize} 
When $M=M'$ we have that
\eqn{\nonumber
	\pr_{\sub{\psi}_t}\Big(M\subseteq E(\PA_t),\ M'\subseteq E(\PA_t)\Big) = \pr_{\sub{\psi}_t}(M\subseteq E(\PA_t)),
}
so that the corresponding contribution in the sum in \eqref{for:subg:variance:2} is 
\eqn{\nonumber 
	\pr_{\sub{\psi}_t}(M\subseteq E(\PA_t))-\pr_{\sub{\psi}_t}(M\subseteq E(\PA_t))^2\leq \pr_{\sub{\psi}_t}(M\subseteq E(\PA_t)),
}
and the sum over $M$ gives the term $\E_{\sub{\psi}_t}[N_t(H)]$ in the statement of Lemma~\ref{lem:bounvariance:subgraph}. 
When $M\neq M'$ and $M\cup M'$ is attainable and their sets of edges are disjoint  it follows directly from the independence of $(U_{v,j})_{v\in[t]}^{j\in[m]}$ and $(\psi_v)_{v\in[t]}$ that
\eqn{\nonumber
	\pr_{\sub{\psi}_t}\Big(M\subseteq E(\PA_t),\ M'\subseteq E(\PA_t)\Big) = \pr_{\sub{\psi}_t}(M\subseteq E(\PA_t))\pr_{\sub{\psi}_t}(M'\subseteq E(\PA_t)).
}
Thus, in this situation the corresponding contribution is zero. 
When $(M,M')$ is not attainable the corresponding contribution is negative. When $(M,M')\in \mathcal{M}$ we bound the corresponding terms in \eqref{for:subg:variance:2} by $\pr_{\sub{\psi}_t}\Big(M\subseteq E(\PA_t),\ M'\subseteq E(\PA_t)\Big)$, thus obtaining
\eqn{
	\mathrm{Var}_{\sub{\psi}_t}(N_t(H))  \leq \E_{\sub{\psi}_t}[N_t(H)]+\sum_{(M,M')\in \mathcal{M}} \pr_{\sub{\psi}_t}\Big(M\cup M' \subseteq E(\PA_t)\Big),
}
 We then rewrite this as
\eqn{
	\mathrm{Var}_{\sub{\psi}_t}(N_t(H,\pi))  \leq \E_{\sub{\psi}_t}[N_t(H)] + \sum_{\hat{H}\in\hat{\mathcal{H}}}\E_{\sub{\psi}_t}[N_t(\hat{H})],
}	
which proves the lemma.
\end{proof}

\subsection{Criterion for conditional convergence}
%The bound given in Lemma~\ref{lem:bounvariance:subgraph} is not sharp. Writing explicitly the conditional second moment of $N_t(H,\pi)$ involves the internal structure of $(H,\pi)$, for which a more detailed analysis is required, thus losing generality. 
%
%We aim for a general criterion for conditional convergence.
%where $\hat{\mathcal{H}}$ is the set of all possible subgraphs $\hat{H}$ that are composed by two copies of $H$ with at least one overlapping edge. Notice that the sets of possible labeled edges that can generate all the subgraphs in $\hat{\mathcal{H}}$ is exactly given by $\mathcal{M}$ as in \eqref{for:overlappingM}. We can now state a simple criterion for the conditional convergence of $N_t(H,\pi)$:

We now prove Proposition~\ref{prop:criterion:conditional:subgraph} using Lemma~\ref{lem:bounvariance:subgraph} and Lemma~\ref{lem:maxpsi}:
\begin{proof}[Proof of Proposition~\ref{prop:criterion:conditional:subgraph}]
It sufficient to show that for every fixed $\varepsilon>0$, 
\eqn{\nonumber 
	\pr\left(|N_{t,\psi}(H,\pi)-\E_t[N_t(H,\pi)]>\varepsilon \E[N_t(H,\pi)] \right) = o(1). 
}
We now apply Lemma~\ref{lem:bounvariance:subgraph}, which yields
\eqn{\nonumber 
\begin{split}
	 \pr\Big(|N_{t,\psi}(H,\pi)-\E_{\sub{\psi}_t}&[N_t(H,\pi)] >\varepsilon \E[N_t(H,\pi)] \Big)  \\
	&  \leq \frac{1}{\varepsilon^2\E[N_t(H,\pi)]^2} \E\big[\mathrm{Var}_{\sub{\psi}_t}(N_t(H,\pi))\big]  \\
	& \leq \frac{\E\Big[\E_{\sub{\psi}_t}[N_t(H,\pi)] + \sum_{\hat{H}\in\hat{\mathcal{H}}}\E_{\sub{\psi}_t}[N_t(\hat{H})]\Big]}{\varepsilon^2\E[N_t(H,\pi)]^2} \\
	&  = \frac{\E[N_t(H,\pi)]+ \E[N_t(\hat{H})]}{\varepsilon^2\E[N_t(H,\pi)]^2} = o(1).
\end{split} 
}
%
%We then proceed to prove the left-hand side of the proposition. Assume that there exists a $\hat{H}\in \hat{\mathcal{H}}$ such that $\Exp{N_t(\hat{H})}=\Theta(\Exp{N_t(H)}^2)$.
\end{proof}

As an example, we consider triangles. Theorem~\ref{th:numberT} identifies the expected number of triangles, and by Theorem~\ref{th:subgraph:expectation} we can show that $\E[\triangle_t^2] = \Theta(\E[\triangle_t]^2)$, so we are not able to apply the second moment method to $\triangle_t$. Figure~\ref{fig:triadfluct} suggests that $\triangle_t/\E[\triangle_t]$ converges to a limit that is not deterministic, i.e., in \eqref{for:subgraph:convergence2} the limiting $X$ is a random variable. 

However, we can prove that $\triangle_t$ is conditionally concentrated, as stated in Corollary~\ref{cor:conc:triang}. The proof of  Corollary~\ref{cor:conc:triang} follows directly from Proposition~\ref{prop:criterion:conditional:subgraph}, the fact that  $\E[\triangle_t] = \Theta(t^{(3-\tau)/(\tau-1)}\log(t))$ as given by Theorem~\ref{th:numberT}, and Figure~\ref{fig:graphlet4tr}, that contains the information on the subgraphs consisting of two triangles sharing one or two edges.

\subsection{Non-concentrated subgraphs}
We now show that for most $\psi$-sequences, the other direction in Proposition~\ref{prop:criterion:conditional:subgraph} also holds. That is, if there exists a subgraph composed of two merged copies of $H$ such that the condition in Proposition~\ref{prop:criterion:conditional:subgraph} does not hold, then for most $\psi$-sequences, $H$ is not conditionally concentrated. 
\begin{Proposition}\label{prop:condconc}
	Consider a subgraph $(H,\pi)$ such that $\Exp{N_t(H,\pi)}\to\infty$ as $t\to\infty$. Suppose that there exists a subgraph $\hat{H}$, composed of two distinct copies of $(H,\pi)$ with at least one edge in common such that $\Exp{N_t(\hat{H})}/\Exp{N_t(H,\pi)}\nrightarrow 0$ as $t\to\infty$. Then, for any $\varepsilon>0$, there exists $\eta>0$ such that
	\begin{equation}
		\Prob{\frac{\Varp{N_t(H,\pi)}}{\Exp{N_t(H,\pi)}^2}>\eta}\geq 1-\varepsilon.
	\end{equation}
\end{Proposition}
To prove Proposition~\ref{prop:condconc} we need a preliminary result on the maximum intensity of the P\'olya urn graph:
\begin{Lemma}
	\label{lem:betagammacoupling}
	For every $\varepsilon>0$ there exists $K = K(\varepsilon)\in\N$ such that
	\eqn{\nonumber
		\pr\bigg(\bigcap_{k\geq K}\Big\{\psi_k\leq \frac{(\log k)^2}{(2m+\delta)k}\Big\}\bigg)\geq 1-\varepsilon.
	}
\end{Lemma}
Lemma~\ref{lem:betagammacoupling} is a part of a more general coupling result between $(\psi_k)_{k\in\N}$ and  a sequence of i.i.d.\ Gamma random variables. We refer to \cite[Lemma 3.2]{berger2014} and  \cite[Lemma 4.10]{hofstad2018+} for more detail.
We now state the lemma we need to prove Proposition~\ref{prop:condconc}:
\begin{Lemma}[Maximum intensity]
	\label{lem:maxpsi}
	For every $\varepsilon>0$ there exists $\omega=\omega(\varepsilon)\in(0,1)$ such that, for every $t\in\N$,  
	\eqn{\nonumber
		\pr\Big(\max_{i\in 2,\ldots, t}\psi_i<\omega\Big)\geq 1-\varepsilon.
	}
\end{Lemma}
\begin{proof}
	Fix $\varepsilon>0$, and consider $K(\varepsilon/2)$ as given by Lemma~\ref{lem:betagammacoupling}. For every $\omega\in(0,1)$ we can write
	\eqn{\label{for:maxpsi:0}
		\pr\Big(\max_{i\in2,\ldots,t}\psi_i<\omega\Big) = \pr\Big(\max_{i\in2,\ldots,K}\psi_i<\omega\Big) \pr\Big(\max_{i\in[t]\setminus[K]}\psi_i<\omega\Big),
	}
	where we used the independence of $\psi_2,\ldots,\psi_t$. If $t>K$ the second term in the right-hand side of  \eqref{for:maxpsi:0} is well defined, otherwise we only have the first term.
	Define,
	\eqn{\nonumber
		\omega_1 = \begin{cases} \frac{(\log K)^2}{(2m+\delta)K} & \mbox{if}\ t>K,\\
			0 & \mbox{if}\ t\leq K.\end{cases}
	}
	Notice that, since the function $k\mapsto \frac{(\log k)^2}{(2m+\delta)k}$ is decreasing, it follows that
	\eqn{
		\label{for:maxpsi:1}
		\pr\Big(\max_{i\in[t]\setminus[K]}\psi_i<\omega_1\Big)\geq 1-\varepsilon/2.
	}
	Define the random variable $X_K = \max_{i\in 2,\ldots,K}\psi_i$, denote its distribution function by $F_K$ and the inverse of its distribution function by $F^{-1}_K$. Consider $\omega_2 = F^{-1}_{K}(1-\varepsilon/2)$, that implies
	\eqn{
		\label{for:maxpsi:2}
		\pr\Big(\max_{i\in[K]}\psi_i<\omega_2\Big)= 1-\varepsilon/2.
	}
	Consider then $\omega = \max\{\omega_1,\omega_2\}$. Using a\eqref{for:maxpsi:1} and \eqref{for:maxpsi:2} with $\omega$ in \eqref{for:maxpsi:0}, it follows that
	\eqn{\nonumber
		\pr\Big(\max_{i\in 2,\ldots, K}\psi_i<\omega\Big) \pr\Big(\max_{i\in[t]\setminus[K]}\psi_i<\omega\Big)\geq 
		(1-\varepsilon/2)^2\geq 1-\varepsilon,
	}
	which completes the proof.
\end{proof}

\begin{proof}[Proof of Proposition~\ref{prop:condconc}]
We use the expression of the conditional variance of~\eqref{for:subg:variance:2}. We first study the term in the conditional variance corresponding to $\hat{H}$. Let $\tilde{\mathcal{M}}$ denote the set of labeled edges $M,M'$ that together form subgraph $\hat{H}$. Let the edges that $M$ and $M'$ share be denoted by $M_s$. Furthermore, let $\tilde{\mathcal{M}}_1$ denote the set of labeled edges $M,M'$ that together form subgraph $\hat{H}$ that do not use vertex 1. We can then write this term as
\begin{equation}
\begin{aligned}[b]
	&\sum_{M,M'\in \tilde{\mathcal{M}}}\Probp{M\cup M'\subseteq E(PA_t)}(1-\Probp{M_s\subseteq E(PA_t)})\\
	& \geq \sum_{M,M'\in \tilde{\mathcal{M}}_1}\Probp{M\cup M'\subseteq E(PA_t)}(1-\psi_{\max})\\
	& = (1-\psi_{\max})\Expd{N_t(\hat{H})}
\end{aligned}
\end{equation}
where the inequality uses~\eqref{for:edgeUrnP}, and $\psi_{\max}=\max_{i\in 2,\ldots, t}\psi_i$. Note that here we excluded vertex 1 from the number of subgraphs with negligible error. By Lemma~\ref{lem:maxpsi} there exists $\omega$ such that with probability at least $1-\varepsilon$, $\psi_{\max}<\omega<1$.

By the assumption on $\hat{H}$, $\Exp{N_t(\hat{H})}\geq \tilde{C}\Exp{N_t(H,\pi)}^2$ for some $\tilde{C}>0$. We the use that $\Expd{N_t(\hat{H})}=O_\pr(\Exp{N_t(\hat{H})})$. Thus, for $t$ sufficiently large,  we can bound the contribution from subgraph $\hat{H}$ to the conditional variance from below with probability at least $1-\varepsilon$ by
	\begin{equation}
	\sum_{M,M'\in \tilde{\mathcal{M}}}\Probp{M\cup M'\subseteq E(PA_t)}(1-\Probp{M_s\subseteq E(PA_t)})\geq C \Expd{N_t(H,\pi)}^2,
	\end{equation}
	for some $C>0$.
	
	Note that the only terms that have a negative contribution to~\eqref{for:subg:variance:2} are the terms where $M\cup M'$ is a non-attainable subgraph. In that situation, $\pr_{\sub{\psi}_t}\Big(M\subseteq E(\PA_t),\ M'\subseteq E(\PA_t)\Big) =0$. Furthermore, the sum over $\pr_{\sub{\psi}_t}(M\subseteq E(\PA_t))\pr_{\sub{\psi}_t}(M'\subseteq E(\PA_t) \leq  \Expd{N_t(H,\pi)}^2/n^2$, since the two subgraphs share at least two vertices. Therefore, the negative terms in the conditional variance scale as most as $\Expd{N_t(H,\pi)}/n^2$. We therefore obtain that with probability at least $1-\varepsilon$,
	\begin{equation}
	\Varp{N_t(H,\pi)}\geq \eta \Expd{N_t(H,\pi)}^2,
	\end{equation}
	for some $\eta>0$, which proves the proposition.
\end{proof}

\section{Discussion}\label{sec:discussion}
In this paper, we investigated the expected number of times a graph $H$ appears as a subgraph of a PAM for any degree exponent $\tau$. We find the scaling of the expected number of such subgraphs in terms of the graph size $t$ and the degree exponent $\tau$ by defining an optimization problem that finds the optimal structure of the subgraph in terms of the ages of the vertices that form subgraph $H$ and by using the interpretation of the PAM as a P\'olya urn graph.  

We derive the asymptotic scaling of the number of subgraphs. For the triangle subgraph, we obtain more precise asymptotics. It would be interesting to obtain precise asymptotics of the expected number of other types of subgraphs as well. In particular, this is necessary to compute the variance of the number of subgraphs, which may allow us to derive laws of large numbers for the number of subgraphs. We show that different subgraphs may have significantly different concentration properties. Therefore, identifying the distribution of the number of rescaled subgraphs for any type of subgraph remains a challenging open problem.

Another interesting extension would be to investigate whether our result still holds for other types of PAMs, for example models that allow for self-loops, or models that include extra triangles.

We further prove results for the number of subgraphs of fixed size $k$, while the graph size tends to infinity. It would also be interesting to let the subgraph size grow with the graph size, for example by counting the number of cycles of a certain length that grows in the graph size.

Finally, we investigate the number of times $H$ appears as a subgraph of a PAM. It is also possible to count the number of times $H$ appears as an \emph{induced} subgraph instead, forbidding edges that are not present in $H$ to be present in the larger graph. It would be interesting to see whether the optimal subgraph structure is different from the optimal induced subgraph structure. 

\paragraph{Acknowledgments.}
We thank Remco van der Hofstad for reading the manuscript.
This work is supported in part by the Netherlands Organisation for Scientific Research (NWO) through the Gravitation {\sc Networks} grant 024.002.003 and NWO TOP grant 613.001.451.

\printbibliography[title=References, heading = bibintoc]
\end{document}